\documentclass[a4paper,18pt]{article}
\usepackage{amsmath}
\usepackage{indentfirst}
\usepackage[french]{babel}
\usepackage{graphics}
\usepackage{tabularx}
\usepackage{caption}
\usepackage{float}
\usepackage{amssymb}
\usepackage{multirow}
\usepackage{fancyhdr}
\usepackage{dsfont}
\usepackage{amsfonts}
\usepackage{amsthm}
\usepackage[dvips]{graphicx}
\usepackage[latin1]{inputenc}
\usepackage{enumerate}
\usepackage{amsxtra}
\usepackage{amstext}
\usepackage{amssymb}
\usepackage{latexsym}
\usepackage[pdftex,colorlinks]{hyperref}
\newcommand{\NN}{\mathbb N}
\newcommand{\EE}{\mathbb E}
\newcommand{\RR}{\mathbb R}

\newcommand{\Ff}{\mathcal L}
\newcommand{\PP}{\mathcal P}
\newcommand{\C}{\mathbb C}
\newcommand{\SL}{\textrm{SL}}
\newcommand{\GL}{\textrm{GL}}
\newcommand{\ZZ}{\mathbb Z}

\newcommand{\R}{\mathbb R}
\newcommand{\E}{\mathbb E}
\newcommand{\Q}{\mathbb Q}
\newcommand{\p}{\mathbb P}
\newcommand{\HD}{\textrm{HD}}
\newtheorem{proposition}{Proposition}[section]
\newtheorem{lemme}[proposition]{Lemme}
\newtheorem{propdef}[proposition]{Proposition/Définition}
\newtheorem{corollaire}[proposition]{Corollaire}
\newtheorem{theorem}[proposition]{Théorème}
\newtheorem{definition}[proposition]{Definition}

\newtheorem{remarque}[proposition]{Remarque}

\newtheorem{exemple}[proposition]{Exemple}

\newcommand{\n}{\mathds{N}}
\newcommand{\Z}{\mathbb{Z}}

\providecommand{\keywords}[1]{\textbf{\textit{Keywords: \,}} #1}
\title{Comptage probabiliste sur la fronti\`{e}re de Furstenberg}
\author{\Large{Aoun Richard\footnote{American University of Beirut, Department of Mathematics, Faculty of Arts and Sciences, P.O. Box 11-0236 Riad El Solh, Beirut 1107 2020, LEBANON
E-mail address: ra279@aub.edu.lb}}}
\date{}
  
\begin{document}
\maketitle
\begin{abstract}Soit $G$ un groupe lin\'{e}aire alg\'{e}brique r\'{e}el semi-simple  sans facteurs compacts et $\Gamma$ un sous-groupe Zariski dense.  Nous nous int\'{e}ressons aux propri\'{e}t\'{e}s asymptotiques de $\Gamma$ vis-\`{a}-vis de
la fronti\`{e}re de Furstenberg de $G$. Nous montrons en un premier temps que
les composantes de $\Gamma$ dans la d\'{e}composition $KAK$ de $G$
deviennent asymptotiquement ind\'{e}pendantes, r\'{e}sultat analogue \`{a}
un th\'{e}or\`{e}me de Gorodnik-Oh  \cite{gorodnik}  pour le comptage des
points de l'orbite d'un r\'{e}seau dans un espace sym\'{e}trique. Par la
suite, nous donnons une nouvelle preuve d'un r\'{e}sultat de
Guivarc'h \cite{Guivarch3} concernant la positivit\'{e} de la
dimension des mesures stationnaires sur la fronti\`{e}re de
Furstenberg. Finalement, nous combinons ces r\'{e}sultats pour
donner une preuve probabiliste de l'alternative de Tits
\cite{tits}, \`{a} savoir deux \'{e}l\'{e}ments g\'{e}n\'{e}riques
de $\Gamma$ engendrent un sous-groupe libre. Ce r\'{e}sultat est
un cas particulier d'un travail ant\'{e}rieur \cite{aoun} et
a r\'{e}pondu \`{a} une question de Guivarc'h \cite{Guivarch3}.
Nous en donnons une preuve plus directe et un \'{e}nonc\'{e} plus g\'{e}n\'{e}ral.
\end{abstract}

\tableofcontents
\section{Introduction}
Soit $G$ un groupe lin\'{e}aire alg\'{e}brique r\'{e}el semi-simple  sans
facteurs compacts et $\Gamma$ un sous-groupe Zariski dense, comme
par exemple $G=\SL_n(\RR)$ et  $\Gamma=\SL_n(\ZZ)$.\\ Il existe,
de fa\c{c}on g\'{e}n\'{e}rale, divers types de comptage permettant d'estimer
le nombre d'\'{e}l\'{e}ments de $\Gamma$ qui appartiennent \`{a} un certain
sous-ensemble $D$ de $G$. Le comptage archim\'{e}dien consiste \`{a}
consid\'{e}rer un $G$-espace homog\`{e}ne $B=G/P$ muni d'une m\'{e}trique
convenable, \`{a} recouvrir $B$
 par des boules $(B_T)_{T\in \RR^+}$ de plus en plus grosses et \`{a} estimer, quand $T$ tend vers $+\infty$,
  le nombre d'\'{e}l\'{e}ments $\gamma$ de $\Gamma$ qui appartiennent \`{a} $D$ et tels que $\gamma \cdot x_0\in B_T$, o\`{u} $x_0$ est un point fix\'{e} de $B$.\\
Si le cas d'un r\'{e}seau $\Gamma$ est bien compris \cite{eskin},
\cite{drs}, \cite{eam}, \cite{ems}, \cite{gorodnik},
  le cas d'un
 ``groupe fin'', i.e. un sous-groupe Zariski dense mais de covolume infini, s'av\`{e}re  beaucoup plus
 d\'{e}licat  et suscite un int\'{e}r\^{e}t croissant pour ses applications en arithm\'{e}tique (voir \cite{sarnaksurvey}).
  Dans \cite{quintcomptage} et \cite{quintens}, Quint traite le cas o\`{u} $\Gamma$ est de type Schottky pour un comptage
  sur l'espace sym\'{e}trique associ\'{e} \`{a} $G$.  R\'{e}cemment, Oh et Shah \cite{ohshah}, \cite{ohsurvey} \'{e}tudient le cas d'un
  sous-groupe $\Gamma$ fin de $\operatorname{SO}(n,1)$.\\

Dans ce texte, nous utilisons un comptage probabiliste ayant
l'avantage de ne  pas faire la diff\'{e}rence entre
 un r\'{e}seau et un groupe fin. En fait, $\Gamma$ n'a m\^{e}me pas besoin d'\^{e}tre discret, ni de type fini; seule l'hypoth\`{e}se  Zariski dense suffit.
  La m\'{e}thode consiste \`{a} consid\'{e}rer une mesure de probabilit\'{e} $\mu$ sur $\Gamma$ dont le support engendre tout le groupe et
  \`{a} \'{e}tudier le comportement de $\mu^n(D)$ quand $n$ tend vers $+\infty$, o\`{u} $\mu^n$ est la $n^{\textrm{\`{e}me}}$ convol\'{e}e de
  $\mu$. Quand $\Gamma$ est de type fini et $\mu$ la probabilit\'{e} uniforme sur une partie g\'{e}n\'{e}ratrice symétrique finie, cela revient \`{a} faire un comptage sur la boule  de centre l'identit\'{e} et de rayon $n$ pour la m\'{e}trique des mots. Les m\'{e}thodes utilis\'{e}es dans notre texte reposent sur la th\'{e}orie des produits
   de matrices al\'{e}atoires \cite{Furst}, \cite{kesten1},   \cite{guivarch},  \cite{Guivarch3}, \cite{guilocal}, \cite{bougerol},
   \cite{page}.\\

  Dans la Section \ref{secasymp}, nous donnons une ind\'{e}pendance asymptotique des parties $K$ des \'{e}l\'{e}ments de $\Gamma$ dans la
  d\'{e}composition $KA^+K$ de $G$ (voir la Section \ref{secalg}). Notre r\'{e}sultat est   parall\`{e}le
   \`{a} un th\'{e}or\`{e}me de comptage archim\'{e}dien de Gorodnik-Oh \cite{gorodnik}  valable uniquement
   pour $\Gamma$ r\'{e}seau (voir le Th\'{e}or\`{e}me \ref{gorodoh}).\\
    \noindent Soit $P$ un parabolique minimal de $G$, $M$ le centralisateur de $A$ dans $K$ et  $B=G/P \simeq K/M$ la fronti\`{e}re de Furstenberg de $G$ (par exemple, pour $G=\SL_n(\RR)$, $B$ est la vari\'{e}t\'{e} des drapeaux complets). Pour tout $g\in G$, notons  $g=k(g)a(g)u(g)$ le r\'{e}sultat dans la d\'{e}composition $KA^+K$. Nous obtenons:
\begin{theorem}(Ind\'{e}pendance asymptotique pour un comptage probabiliste)\\
Soit $\Gamma$ un sous-groupe Zariski dense de $G$ muni d'une mesure de probabilit\'{e}  $\mu$ dont le support engendre $\Gamma$. On suppose que $\mu$ a un moment exponentiel\footnote{Par exemple, si $\Gamma$ est de type fini et le support de $\mu$ est une partie g\'{e}n\'{e}ratrice de $\Gamma$.} (D\'{e}finition \ref{defimoment}).  On note $\nu$ (resp. $\nu^*$) l'unique  mesure  de probabilit\'{e} $\mu$-invariante  d\'{e}finie  sur la fronti\`{e}re de Furstenberg $B=G/P$ (resp. $B^*=P\backslash G$) (voir le Th\'{e}or\`{e}me \ref{unicite}).  Alors, il existe $\rho\in ]0,1[$ tel que pour toute fonction $\phi$   lipschitzienne\footnote{voir la D\'{e}finition \ref{distfront} pour la m\'{e}trique utilis\'{e}e}  sur $B\times B^*$ de constante de Lipschitz  \textrm{Lip($\phi$)} et pour tout entier $n$ assez grand:
$$\Big|\int_{\Gamma} {\phi\Big(k(g) M,M u(g)\Big) \,d\mu^n{(g)} \;  } - \int_{B \times B^*} {  \phi \,d{({\nu}\otimes {\nu^*})}\; } \Big|\leq \textrm{Lip($\phi$)}\; \rho^n$$
\label{indasymp}\end{theorem}
 Des r\'{e}sultats  similaires pour la d\'{e}composition d'Iwasawa ont \'{e}t\'{e}
 d\'{e}montr\'{e}s par Guivarc'h dans son travail des ann\'{e}es 1990: \cite[Th\'{e}or\`{e}me 6', Lemme 8]{Guivarch3}.
  Notre travail en est fortement inspir\'{e}.\\
\noindent La pr\'{e}sence de $M$ dans cet \'{e}nonc\'{e} est naturelle car $k(g)$ (resp. $u(g)$) est unique modulo $M$ pour les classes
\`{a} gauches (resp. pour les classes \`{a} droite).
  La preuve est divis\'{e}e en deux grandes \'{e}tapes.  La premi\`{e}re est
la convergence presque s\^{u}re des parties $K$ avec vitesse exponentielle: Th\'{e}or\`{e}me \ref{convkak}. Si la convergence est classique \cite{bougerol}, la vitesse de convergence est par contre moins connue.
 Cette derni\`{e}re a \'{e}t\'{e} trait\'{e}e dans un travail ant\'{e}rieur \cite{aoun}.
  Cependant, le fait que nous travaillons uniquement dans le corps des r\'{e}els  nous permet de donner une  preuve
   plus directe qui ne repose pas  sur  la d\'{e}composition d'Iwasawa mais sur des  arguments de Goldsheid-Margulis
   dans leur preuve du th\'{e}or\`{e}me d'Oseledets \cite{Margulis}.\\
La deuxi\`{e}me partie est le Lemme \ref{lemmeind}: r\'{e}sultat g\'{e}n\'{e}ral d'ind\'{e}pendance asymptotique valable sur tout groupe lin\'{e}aire ou non.
Malgr\'{e} la simplicit\'{e} des arguments utilis\'{e}s,  ce lemme pourrait avoir une port\'{e}e g\'{e}n\'{e}rale pour
 s\'{e}parer les points attractifs et r\'{e}pulsifs d'un groupe de type
 Schottky.\\

Dans la Section \ref{sechausdorff}, nous donnons une nouvelle preuve d'un r\'{e}sultat de Guivarc'h qui montre  la positivit\'{e}
de la dimension de Hausdorff des mesures stationnaires et donc de l'ensemble limite de $\Gamma$. Notre preuve n'utilise pas la d\'{e}composition d'Iwasawa. Elle repose uniquement sur la convergence exponentielle de la marche al\'{e}atoire sur la fronti\`{e}re (Th\'{e}or\`{e}me \ref{convdirection}).
\begin{theorem}\cite[Th\'{e}or\`{e}me 7']{Guivarch3} (Positivit\'{e} de la dimension de Hausdorff)\\
Soit $\Gamma$ un sous-groupe Zariski dense de $G$ muni d'une mesure de probabilit\'{e}  $\mu$ dont le support engendre $\Gamma$. On suppose que $\mu$ a un moment exponentiel (D\'{e}finition \ref{defimoment}).  On note $\nu$ l'unique mesure $\mu$-invariante sur $B$. Alors $\HD(\nu)>0$. Ici, $\HD(\nu)$ est la dimension de Hausdorff de $\nu$ (voir la Section \ref{sechausdorff}). \label{hausdorff}\end{theorem}

Dans la derni\`{e}re partie du texte (Section \ref{secping}),  nous combinons les r\'{e}sultats pr\'{e}c\'{e}dents pour donner
 une version probabiliste de
l'alternative de Tits \cite{tits}, \`{a} savoir: deux marches
al\'{e}atoires ind\'{e}pendantes sur deux sous-groupes $\Gamma_1$ et
$\Gamma_2$ Zariski denses dans $G$
 engendrent un sous-groupe libre avec une probabilit\'{e} tendant vers $1$ exponentiellement vite. Dans \cite{aoun}, ce r\'{e}sultat a \'{e}t\'{e} d\'{e}montr\'{e} pour $\Gamma_1=\Gamma_2$  de type fini, non virtuellement r\'{e}soluble et d'adh\'{e}rence de Zariski d\'{e}finie sur un corps quelconque.
 Nous profitons que nous travaillons dans le corps des r\'{e}els pour
all\'{e}ger les hypoth\`{e}ses "de type fini" et
\emph{$\Gamma_1=\Gamma_2$} et  pour donner des preuves plus
directes et simplifi\'{e}es. Le r\'{e}sultat obtenu est le suivant:
\begin{theorem} Soient $\Gamma_1$ et $\Gamma_2$ deux sous-groupes Zariski denses de $G$ munis chacun  d'une mesure de probabilit\'{e}  qu'on notera respectivement $\mu_1$ et $\mu_2$. On suppose qu'elles poss\`{e}dent un moment exponentiel et que leur support engendre respectivement $\Gamma_1$ et $\Gamma_2$. Alors il existe $\rho\in ]0,1[$ tel que pour tout entier $n$ assez grand:
$$\mu_1^n \otimes \mu_2^n\{(x,y)\in \Gamma_1\times \Gamma_2; \langle x, y \rangle \textrm{\;est libre} \} \geq 1 - \rho^n$$
\label{titsproba}\end{theorem}

Signalons finalement que la Section \ref{secalg}  rassemble des r\'{e}sultats classiques sur la th\'{e}orie des groupes alg\'{e}briques semi-simples r\'{e}el. Ils permettant de ramener l'\'{e}tude sur la fronti\`{e}re \`{a} celle d'un nombre fini de repr\'{e}sentations
irr\'{e}ductibles proximales (i.e. contenant un \'{e}l\'{e}ment contractant l'espace projectif). Le lecteur non familier avec cette th\'{e}orie peut se contenter de traiter $G=\SL_d(\RR)$ car le cas g\'{e}n\'{e}ral se ram\`{e}ne dans une large mesure au cas de $\SL_d(\RR)$, gr\^{a}ce \`{a} la Proposition \ref{mostowcons}.

\paragraph{Remerciements:} Ce travail a \'{e}t\'{e} effectu\'{e} gr\^{a}ce au soutien de l'ERC 208091-GADA.
Je remercie l'universit\'{e} d'Orsay et plus sp\'{e}cialement Emmanuel Breuillard pour
des conditions de travail exceptionnelles. Ce texte vient \`{a} la
suite de plusieurs expos\'{e}s effectu\'{e}s dans le cadre du GDR
Platon et a \'{e}t\'{e} mis au point pour la conf\'{e}rence en
l'honneur d'Emile Le Page \`{a} l'\^{\i}le de Berder en septembre
2011. Je remercie sp\'{e}cialement Fran\c{c}oise Dalbo pour ces
exceptionnelles rencontres. Je remercie \'{e}galement Bertrand
Deroin, Emile Le Page, Sara Brofferio et Yves Guivarc'h pour leur
disponibilit\'{e} et de nombreuses discussions.
\section{Pr\'{e}liminaires alg\'{e}briques}\label{secalg}
Dans ce texte, un groupe  lin\'{e}aire alg\'{e}brique r\'{e}el $G$ d\'{e}signe un
sous-groupe de $\operatorname{GL}_d(\RR)$, pour un certain entier
$d\geq 2$, qui soit ferm\'{e} pour la topologie de Zariski\footnote{En d'autres termes, $G$ est le lieu d'annulation d'une famille finie de polynômes en $d^2$ variables et \`{a} coefficients réels.}. Un tel
groupe appara\^{i}t de fa\c{c}on naturelle  dans nos \'{e}nonc\'{e}s comme
l'adh\'{e}rence de Zariski\footnote{L'adh\'{e}rence de Zariski d'un sous-groupe
$\Gamma$ de $\GL_d(\RR)$ est le plus petit sous-groupe lin\'{e}aire
alg\'{e}brique r\'{e}el $G$ contenant $\Gamma$.} dans $\GL_d(\RR)$ d'un sous-groupe $\Gamma$. Le groupe $G$ est dit semi-simple s'il est Zariski connexe et n'admet
pas de sous-groupes normaux Zariski connexes ab\'{e}liens diff\'{e}rents de $\{1\}$.

Une repr\'{e}sentation rationnelle $\rho: G \longrightarrow \GL(V)$ de
$G$ d\'{e}signera un morphisme de groupes de $G$ dans $\GL(V)$, o\`{u}
$V$ est un espace vectoriel r\'{e}el, qui soit aussi un morphisme de
$\RR$-vari\'{e}t\'{e}s alg\'{e}briques.
\subsection{La d\'{e}composition $KA^+K$}\label{defi2}
Soit $G$ un groupe lin\'{e}aire alg\'{e}brique
r\'{e}el semi-simple. Quand $G$ est non compact, $G$ admet un tore maximal d\'{e}ploy\'{e} non trivial $A$, i.e. un sous-groupe alg\'{e}brique r\'{e}el ayant une repr\'{e}sentation rationnelle fid\`{e}le dans le groupe des matrices diagonales inversibles \`{a} coefficients réels. On note $X$ l'ensemble des caract\`{e}res rationnels de $A$ (c'est un $\ZZ$-module libre) et $E$ l'espace vectoriel r\'{e}el $X \otimes_{\ZZ} \RR$. On note $\Sigma$ l'ensemble des racines de $A$ dans $G$, i.e. les poids non triviaux de $A$ dans la repr\'{e}sentation adjointe de $G$. On peut montrer
 que $\Sigma$  est un syst\`{e}me de racines de $E$ \cite{bt}. On choisit alors un syst\`{e}me de racines positives $\Sigma^+$ et on note $A^+=\{a\in A; |\chi(a)|\geq 1 \;\forall \chi\in \Sigma^+\}$. On a alors:
\begin{theorem}\cite[Chap. 9, Th. 1.1]{helg}
Soit $G$ un groupe lin\'{e}aire alg\'{e}brique  r\'{e}el semi-simple et non
compact. Il existe un compact maximal $K$ de $G$ tel que $G=K
A^+K$. De plus, si $M$ est le centralisateur de $A$ dans $K$ et
$g=k(g)a(g)u(g)=k'(g)a(g)u'(g)$ sont deux d\'{e}compositions  de $g\in G=KA^+K$ avec $a(g)$ dans l'int\'{e}rieur de $A^+$,
alors il existe $m\in M$ tel que $k'(g)=k(g)m$ et
$u'(g)=m^{-1}u(g)$. \label{theokak}\end{theorem}
\begin{exemple}
Pour $G=\SL_d(\RR)$, un compact maximal $K$  est le sous-groupe des matrices orthogonales.  La d\'{e}composition $G=KA^+K$ n'est autre que la d\'{e}composition polaire classique, i.e. la d\'{e}composition d'une matrice inversible en un produit d'une matrice orthogonale et d'une matrice sym\'{e}trique d\'{e}finie positive.\end{exemple}

La proposition classique suivante montre que si $(\rho, V)$ est une
repr\'{e}sentation rationnelle de $G$, alors la d\'{e}composition $KAK$ de
$\rho(G)$ mime celle de $\operatorname{SL}(V)$ d\'{e}crite dans
l'exemple pr\'{e}c\'{e}dent. Pour la simplicit\'{e} de l'expos\'{e}, nous omettons la preuve et renvoyons par exemple \`{a} \cite[\S 4.2]{aoun1}.
\begin{proposition}
Soit $G$ un groupe lin\'{e}aire alg\'{e}brique r\'{e}el semi-simple sans facteurs compacts et $(\rho,V)$
une repr\'{e}sentation irr\'{e}ductible   rationnelle de $G$. Il existe un produit scalaire $\langle \cdot ,  \cdot
\rangle$ sur $V$ et une base orthonorm\'{e}e $F$  telles que
$\rho(A^+)\subset \{diag\left(a_1,\cdots, a_{dim(V)}\right);
|a_1|\geq |a_i|\;\forall i \neq 1\}$ et $\rho(K) \subset \{g\in
\operatorname{SL}(V); g g^t = Id\}$. En particulier, pour tout
$g\in G$, si $a(g)$ est la composante dans $A^+$ de $g$ dans la
d\'{e}composition $KA^+K$ et $a_i\left(\rho(g)\right)$ la
$i^{\textrm{\`{e}me}}$ composante de la matrice diagonale
$\rho\left(a(g)\right)$ \'{e}crite dans la base $F$, on a:
\begin{equation}
{||\rho(g)||}= a_1\left(\rho(g)\right)\;\;\;\;;\;\;\;\, \frac{||\bigwedge^2 \rho(g)||}{||\rho(g)||^2}= \underset{j\neq 1}{\sup} \frac{a_j\left(\rho(g)\right)}{a_1\left(\rho(g)\right)}
\label{eqmostow}\end{equation}
\label{mostowcons}\end{proposition}
\begin{remarque} Pour $G=\SL_d(\RR)$ et $\rho$ la représentation naturelle de $G$ sur $\RR^d$, il suffit de prendre la base
 et le produit scalaire canoniques. \end{remarque}
\subsection{Fronti\`{e}re de Furstenberg}
Dans cette section $G$ d\'{e}signe un groupe alg\'{e}brique r\'{e}el semi-simple sans facteurs compacts.
\begin{definition}(Fronti\`{e}re de Furstenberg)
Soit $P$ un sous-groupe parabolique\footnote{i.e. un sous-groupe alg\'{e}brique r\'{e}el de $G$ tel que l'espace homog\`{e}ne $G/P$ soit une vari\'{e}t\'{e} projective réelle.} minimal  de $G$. On appelle fronti\`{e}re de Furstenberg de $G$  l'espace homog\`{e}ne
$B=G/P$ pour les classes \`{a} gauche modulo $P$. De
m\^{e}me on note $B^*=P\backslash G$ l'ensemble quotient de $G$
par $P$ pour les classes \`{a}
droite. On peut montrer  que $B$ s'identifie \`{a} $K/M$ et $B^*$ \`{a} $M\backslash K$, $M$ \'{e}tant le centralisateur de $A$ dans $K$.
\label{deffront}\end{definition}
\begin{exemple}
Pour $G=\operatorname{SL_d(\RR)}$, le sous-groupe des matrices triangulaires sup\'{e}rieures est un parabolique minimal et $B$ est la vari\'{e}t\'{e}
des drapeaux complets.
\end{exemple}
Une notion  cruciale dans ce texte est celle de proximalit\'{e}:
\begin{definition}
Un \'{e}l\'{e}ment $g\in \operatorname{GL}_d(\RR)$ est dit proximal s'il admet une unique valeur propre de module maximal. Une repr\'{e}sentation lin\'{e}aire $(\rho, V)$ d'un groupe $G$ est dite proximale si $\rho(G)$ contient un \'{e}l\'{e}ment proximal.
\label{defproxx}\end{definition}

La remarque suivante et le Lemme \ref{distfront} expliquent
pourquoi l'\'{e}tude d'un groupe semi-simple sans facteurs compacts peut se ramener,
dans beaucoup de cas,  \`{a} celle d'une repr\'{e}sentation irr\'{e}ductible
proximale.
\begin{remarque}
Soit $H$ un sous-groupe  de $\GL_d(\RR)$. Il est clair que quand
$H$ est relativement compact,  aucune repr\'{e}sentation de $H$ n'est
proximale. R\'{e}ciproquement, si $H$ est semi-simple sans facteurs
compacts, toute repr\'{e}sentation irr\'{e}ductible $(\rho,V)$ de $H$ peut
\^{e}tre rendue proximale dans le sens suivant: il existe un entier
$r<d$ et une sous-repr\'{e}sentation irr\'{e}ductible proximale  de la repr\'{e}sentation
produit ext\'{e}rieur  $\bigwedge^r V$. En effet,
on montre \`{a} l'aide du lemme de Burnside que l'irr\'{e}ductibilit\'{e} de
$\rho$ et la non compacit\'{e} de $\rho(H)$ entraînent l'existence
d'un \'{e}l\'{e}ment $\gamma$ de $H$, tel que $\rho(\gamma)$ n'a pas
toutes les valeurs propres de m\^{e}me module.
\label{ani}\end{remarque}
\begin{lemme}(La fronti\`{e}re de Furstenberg comme une vari\'{e}t\'{e} projective)
 Il existe des
repr\'{e}sentations irr\'{e}ductibles rationnelles et proximales
$(\rho_i,V_i)_{i=1}^r$ de $G$ et des vecteurs $(v_{i})_{i=1}^r$
telles que les applications suivantes sont injectives:
$$ \begin{array}{ccc}
G/P & \overset{\pi_1}{\longrightarrow} &{\prod}_{i=1}^r{P(V_i)}\\
gP & \longmapsto & (g\cdot [v_{i}])_{i=1}^r\end{array} ~~~~~~\textrm{et}~~~~~~
\begin{array}{ccc}
P\backslash G& \overset{\pi_2}{\longrightarrow} &{\prod}_{i=1}^r{P(V_i^*)}\\
Pg & \longmapsto & (g^{-1}\cdot [v_{i}^*])_{i=1}^r\end{array}$$
Ici $[x]$ d\'{e}signe le projet\'{e} d'un vecteur $x\in
V\setminus\{0\}$ dans $V$.
\label{furstprojective}\end{lemme}
\noindent Nous r\'{e}f\'{e}rons par
exemple \`{a} \cite[\S II.2.6]{quint} o\`{u} de tels
plongements sont utilis\'{e}s.
\begin{exemple}
    Pour $G=\operatorname{SL}_d(\RR)$, les repr\'{e}sentations $V_i$ sont les repr\'{e}sentations produits ext\'{e}rieurs de la repr\'{e}sentation naturelle et les vecteurs $v_{\rho_i}$ sont les $e_1\wedge \cdots \wedge e_i$, o\`{u} $(e_1,\cdots, e_n)$ est la base canonique de $\RR^d$  \end{exemple}
\begin{definition}(Distance de Fubini-Study et m\'{e}trique sur la fronti\`{e}re)
\begin{enumerate}
\item  Soit $V$ un espace vectoriel de dimension finie et $||\cdot||$ une norme sur $V$.  La distance de Fubini-Study sur l'espace projectif $P(V)$ est d\'{e}finie par:
$$\delta([x],[y]) = \frac{||x \wedge y||}{||x||\;||y||}$$
\item Soit $G$ un groupe lin\'{e}aire
alg\'{e}brique r\'{e}el semi-simple, $B$ sa fronti\`{e}re de Furstenberg et $B
\hookrightarrow {\prod}_{i=1}^r{P(V_i)}$ le plongement donn\'{e} par
le Lemme \ref{furstprojective}. Sur chaque espace projectif
$P(V_i)$, on consid\`{e}re la distance de Fubini-Study $\delta_i$ pour
la norme sur $V_i$ induite par le produit scalaire $K$-invariant
de la Proposition \ref{mostowcons}.  La m\'{e}trique choisie sur $B$
est celle induite par la distance $\max\{\delta_i, i=1,\cdots,
r\}$. On d\'{e}finit de fa\c{c}on  analogue une m\'{e}trique sur $B^*$.
\end{enumerate}\label{distfront} \end{definition}
\section{Marches al\'{e}atoires sur les groupes Zariski denses des groupes lin\'{e}aires alg\'{e}briques semi-simples}\label{secgeneral}
Cette partie contient les piliers de la th\'{e}orie des produits de matrices al\'{e}atoires.  Dans la premi\`{e}re section,  nous exposons deux th\'{e}or\`{e}mes fondamentaux dus \`{a} Guivarc'h et Raugi dans leur travaux des ann\'{e}es 1980-1990. Il s'agit de l'unicit\'{e} de la mesure stationnaire sur la fronti\`{e}re (Th\'{e}or\`{e}me \ref{unicite}) et la s\'{e}paration des exposants de Lyapunov (Th\'{e}or\`{e}me \ref{interieur}). Dans la deuxi\`{e}me, nous montrons comment l'utilisation des cocycles dans le monde non commutatif permet d'obtenir des vitesses exponentielles, en particulier le Th\'{e}or\'{e}me \ref{convdirection}. Ce dernier montre que la marche al\'{e}atoire converge presque s\^{u}rement dans la fronti\`{e}re de Furstenberg avec vitesse exponentielle. Commen\c{c}ons par des d\'{e}finitions usuelles:

\begin{definition}(Marches al\'{e}atoires et mesures stationnaires)\\
Soit $\Gamma$ un groupe et $\mu$ une mesure de probabilit\'{e} sur $\Gamma$.
\begin{enumerate}
\item  Une mesure de probabilit\'{e} $\mu$ sur $\Gamma$ est dite adapt\'{e}e si $\Gamma$ est le groupe engendr\'{e} par son support.
\item On consid\`{e}re sur un m\^{e}me espace probabilis\'{e} $(\Omega, \Ff, \PP)$ une suite $(g_i)_{i\in \NN^*}$ de variables al\'{e}atoires ind\'{e}pendantes et de m\^{e}me loi $\mu$ et on note $$x_n=g_1\cdots g_n$$ la marche al\'{e}atoire \`{a} droite et $y_n=g_n\cdots g_1$ celle \`{a} gauche. La loi de $x_n$ est $\mu^n$, la convol\'{e}e $n^{\textrm{\`{e}me}}$ de $\mu$. Le symbole $\EE$ d\'{e}signera l'esp\'{e}rance par rapport \`{a} la probabilit\'{e} $\PP$.
   \item Si $\Gamma$ agit sur un espace $X$, on dit qu'une mesure de probabilit\'{e}  $\nu$ sur $X$ est  $\mu$-invariante ou stationnaire si pour toute fonction r\'{e}elle mesurable $f$ d\'{e}finie sur $X$, on a l'\'{e}galit\'{e} $\iint_{E \times X}{f(g\cdot x) \,d\mu(g) d\nu(x)}=\int_{X}{f(x)\,d\nu(x)}$.\end{enumerate}
\label{defgen}\end{definition}

\subsection{Les travaux de Guivarc'h et Raugi }
Pour \'{e}tudier les marches al\'{e}atoires sur un sous-groupe $\Gamma$ de $\operatorname{GL}_d(\RR)$, Guivarc'h et Raugi supposaient que $\Gamma$ est fortement irr\'{e}ductible\footnote{i.e. tout sous-groupe d'indice fini de $\Gamma$ est irr\'{e}ductible. Quand l'adh\'{e}rence de Zariski de $\Gamma$ est Zariski connexe, les notions d'irr\'{e}ductibilit\'{e} et de forte irr\'{e}ductibilit\'{e} co\"{\i}ncident.} et proximal (voir la D\'{e}finition \ref{defproxx}).  Ces conditions sont g\'{e}n\'{e}ralement abr\'{e}g\'{e}es par $i-p$. La condition $i-p$ est g\'{e}n\'{e}rique dans le sens o\`{u} l'ensemble des mesures de probabilit\'{e} sur $\operatorname{GL}_d(\RR)$ dont le support engendre un sous-groupe satisfaisant les conditions $i-p$  est un ouvert dense pour la topologie de la convergence \'{e}troite des mesures de probabilit\'{e}. \\

En 1985, Goldsheid et Margulis d\'{e}montrent le th\'{e}or\`{e}me suivant qui r\'{e}duit la v\'{e}rification des propri\'{e}t\'{e}s $i-p$ d'un sous-groupe $\Gamma$ de $\operatorname{GL}_d(\RR)$ \`{a} celle de son adh\'{e}rence de Zariski $G$ dans $\GL_d(\RR)$.  Notons que ce th\'{e}or\`{e}me a \'{e}t\'{e} red\'{e}montr\'{e} par Benoist-Labourie \cite{benlab} et Prasad  \cite{prasad}.
\begin{theorem}\cite[Th\'{e}or\`{e}me 6.3]{Margulis} Soit $\Gamma$ un sous-groupe de $\operatorname{GL}_d(\RR)$ qui agit de fa\c{c}on irr\'{e}ductible.
Alors $\Gamma$ est $i-p$ si et seulement si son adh\'{e}rence de Zariski $G$ dans $\GL_d(\RR)$ est $i-p$.\label{goldmarg}\end{theorem}
\begin{remarque} Le th\'{e}or\`{e}me pr\'{e}c\'{e}dent n'est plus vrai si l'on travaille dans un corps local diff\'{e}rent de $\RR$.
Nous r\'{e}f\'{e}rons \`{a} \cite{modern} pour une \'{e}tude pr\'{e}cise de la
condition $i-p$ dans ce cas et la notion de semi-groupes "larges".
\end{remarque} \noindent  Si l'adh\'{e}rence de Zariski $G$ de
$\Gamma$ est $i-p$, alors $G$ est forc\'{e}ment r\'{e}ductif et son
commutateur   est semi-simple et sans facteurs compacts  (voir
\cite[Lemme 2.7]{guivu}). Dans notre texte, nous imposons, pour
simplifier, la condition plus restrictive que $G$ est semi-simple
et sans facteurs compacts. La remarque \ref{ani}  et le Lemme
\ref{furstprojective} permettent alors de ramener l'\'{e}tude  de la
fronti\`{e}re de Furstenberg de $G$ \`{a} celle d'un nombre fini de
repr\'{e}sentations irr\'{e}ductibles proximales.   Ainsi, nous pouvons
formuler les \'{e}nonc\'{e}s des travaux de Guivarc'h et Raugi de la fa\c{c}on
suivante:
\begin{theorem}\cite[Th\'{e}or\`{e}me 2.6]{guivarch}(Unicit\'{e} de la mesure stationnaire)\\
Soit $\Gamma$ un sous-groupe Zariski dense d'un groupe lin\'{e}aire
alg\'{e}brique r\'{e}el $G$, suppos\'{e} semi-simple et sans facteurs
compacts. On note $B$ sa fronti\`{e}re de Furstenberg (D\'{e}finition
\ref{deffront}).
 Alors pour toute mesure de probabilit\'{e} $\mu$ adapt\'{e}e sur $\Gamma$ (voir la Définition \ref{defgen}), il existe une unique mesure stationnaire $\nu$ sur la fronti\`{e}re $B$. De plus, il existe une variable al\'{e}atoire $Z\in B$ de loi $\nu$ telle que, presque s\^{u}rement, la marche al\'{e}atoire \`{a} droite $x_n$ v\'{e}rifie:
\begin{equation}x_n \nu \overset{\textrm{\'{e}troite}}{\underset{n\rightarrow +\infty}{\longrightarrow}}  \delta_{Z} \label{11}\end{equation}
En fait,  si $(\rho,V)$ est une repr\'{e}sentation irr\'{e}ductible proximale de $G$, alors il existe une unique mesure de probabilit\'{e} $\mu$-invariante $\nu$ sur l'espace projectif  $P(V)$. Une propri\'{e}t\'{e} similaire \`{a} (\ref{11}) est valable.
\label{unicite}\end{theorem}

\begin{remarque}
\begin{enumerate}
\item De fa\c{c}on similaire, en notant $\mu^{-1}$ la loi de $g_1^{-1}$, il existe une unique mesure de probabilit\'{e} $\mu^{-1}$-invariante $\nu^*$ sur $B^*=P \backslash G$. Le Lemme \ref{furstprojective} justifie l'emploi de $\mu^{-1}$.
\item Si $\mu$ poss\`{e}de  un moment exponentiel (Definition \ref{defimoment}), alors la mesure de probabilit\'{e} $\nu$ peut \^{e}tre remplac\'{e}e  par la mesure de Dirac $\delta_{x}$ pour tout $x\in B$ dans (\ref{11}). En d'autres termes, $x_n \cdot x$ converge presque s\^{u}rement vers $Z$ \cite{bougerol}. Nous montrons dans le Th\'{e}or\`{e}me \ref{convdirection} que la vitesse de convergence est exponentielle.
\item Dans le m\^{e}me article, les auteurs montrent que la composante gauche de $x_n$ suivant $K$  dans la d\'{e}composition $KA^+K$ de $G$ converge presque s\^{u}rement vers $Z$, une fois projet\'{e}e sur la fronti\`{e}re. Nous montrons dans le Th\'{e}or\`{e}me \ref{convkak} que cette convergence est exponentielle.\end{enumerate}
\end{remarque}
\begin{definition}
Soit $\mu$ une mesure de probabilit\'{e} sur $\operatorname{SL}_d(\RR)$ et $||\cdot||$ une norme sur $\RR^d$.
On dit que $\mu$ a un moment d'ordre $1$ (resp. d'ordre exponentiel) s'il existe $\tau>0$ tel que $\int{\log {||g||}\,d\mu(g)}< \infty$ (resp. $\int{||g||^\tau\,d\mu(g)}< \infty$).
\label{defimoment}\end{definition}
\begin{remarque} Il est facile de montrer (voir \cite[Lemme 4.25]{aoun}) que si $G$ est un sous-groupe de $\operatorname{SL}_d(\RR)$ et $\rho: G \longrightarrow \operatorname{SL}_{d'}(\RR)$ est une repr\'{e}sentation de $G$, alors pour toute mesure de probabilit\'{e} $\mu$ sur $G$ ayant un moment d'ordre $1$ (resp. d'ordre exponentiel), la mesure de probabilit\'{e} image $\rho(\mu)$ poss\`{e}de la m\^{e}me propri\'{e}t\'{e}. \end{remarque}
\begin{propdef}(Exposants de Lyapunov)]
Soit $\mu$ une mesure de probabilit\'{e} sur $\operatorname{SL}_d(\RR)$ ayant un moment d'ordre $1$ et $x_n$ la marche al\'{e}atoire \`{a} droite associ\'{e}e (D\'{e}finition \ref{defgen}). On d\'{e}finit les exposants de Lyapunov de $\mu$ par les limites presque s\^{u}res suivantes: $\lambda_i=\underset{n\rightarrow \infty}{\lim} \frac{1}{n}  \log|| \bigwedge^i x_n||$, $i=1,\cdots, d$. L'existence\footnote{Il est clair par le lemme sous-additif classique que $\lambda_i= \underset{n\rightarrow \infty}{\lim}\frac{1}{n}\EE ( \log|| \bigwedge^i x_n||)$.} de ces limites peut \^{e}tre d\'{e}montr\'{e}e par le th\'{e}or\`{e}me ergodique sous-multiplicatif de Kingman \cite{kingman}. Quand $\rho$ est une repr\'{e}sentation lin\'{e}aire de $G$, les exposants de Lyapunov de la mesure de probabilité image $\rho(\mu)$ sont not\'{e}s $\lambda_i(\rho)$. \label{defprop}\end{propdef}
\begin{remarque}[Loi des grands nombres] L'existence des limites  ci-dessus peut \^{e}tre vue comme une loi des grands nombres. En effet, soient $K=\operatorname{SO}_d(\RR)$ et $d$ la m\'{e}trique $K$-invariante sur l'espace sym\'{e}trique $\operatorname{SL}_d(\RR)/K$ d\'{e}finie par $d(gK,g'K)=\log ||g'^{-1} g||$. Alors $\lambda_1=\underset{n\rightarrow \infty}{\lim} \frac{d(x_nK,K)}{n}$. Les \'{e}quivalents des th\'{e}or\`{e}mes de probabilit\'{e}s classiques comme le th\'{e}or\`{e}me central limite, les grandes d\'{e}viations, le th\'{e}or\`{e}me limite local ont \'{e}t\'{e} d\'{e}montr\'{e}s par Le Page \cite{page}. \label{LGN}\end{remarque}
Nous sommes maintenant en mesure d'\'{e}noncer un th\'{e}or\`{e}me fondamental de Guivarc'h et Raugi concernant la s\'{e}paration des exposants de Lyapunov.
\begin{theorem}\cite[Th\'{e}or\`{e}me 3.5]{guivarch} On consid\`{e}re la m\^{e}me situation qu'au Th\'{e}or\`{e}me \ref{unicite} et on suppose de plus que $\mu$ a un moment d'ordre $1$. Alors pour toute repr\'{e}sentation irr\'{e}ductible et proximale $(\rho,V)$ de $G$,  $\lambda_1(\rho)>\lambda_2(\rho)$. \label{interieur}\end{theorem}
\subsection{Cocyle $\Rightarrow$ Vitesse exponentielle}
\paragraph{Pourquoi les cocycles?}
Soit $G$ un groupe agissant sur un espace $X$. Un cocycle sur $G\times X$ est une application $s: G \times X \longrightarrow \RR$ telle que $s(g_1g_2,x)=s(g_1,g_2\cdot x)+s(g_2,x)$ pour tous $g_1,g_2\in G$ et tout $x\in X$.\\
Les cocycles permettent de transporter la structure multiplicative du groupe et la "transforment" en une somme.
Pour les marches al\'{e}atoires dans le monde non commutatif, cela sert \`{a} mimer le cas commutatif en transformant les quantit\'{e}s en des sommes de variables al\'{e}atoires, sous le prix de perdre l'ind\'{e}pendance\footnote{
Un exemple classique est $G=\operatorname{GL}_d(\RR)$, $X=P(\RR^d)$ et $f(g)=\log ||g||$. Pour obtenir  des th\'{e}or\`{e}mes limites pour $f(g_1\cdots g_n)$, un cocycle naturel \`{a} consid\'{e}rer est $s(g,[x])=\log \frac{||g x||}{||x||}$. C'est un des points de d\'{e}part dans la preuve des th\'{e}or\`{e}mes limites d\^{u} \`{a} Le Page \cite{page}  et expliqu\'{e}s dans la Remarque \ref{LGN}.}.\\

Le lemme de cocycle suivant est  d\^{u} \`{a} Le Page, plus pr\'{e}cis\'{e}ment dans la preuve du Th\'{e}or\`{e}me 1 de \cite{page}. Le lecteur peut aussi voir \cite{bougerol}.
\begin{lemme}(Un lemme de cocycle: premi\`{e}re version \footnote{Une deuxi\`{e}me version sera n\'{e}cessaire pour la Section \ref{secgeneral} et traitera le cas $l=0$.}) \label{cocycle}
Soit $G$ un semi-groupe topologique agissant sur un espace topologique $X$, $s$ un cocycle additif sur $G \times X$, $\mu$ une mesure de probabilit\'{e} qui satisfait la condition suivante: pour $r(g)=sup_{x\in X}{|s(g,x)|}$, il existe $\tau>0$ tel que $\label{condition}\EE\left(exp(\tau r(x_1))\right)< \infty$. On pose $l=\underset{n\rightarrow \infty}{\lim} \;\frac{1}{n}\sup_{x\in X}\;{\EE(s(x_n,x))}$. Si $l<0$, alors il existe $\lambda>0$, $\epsilon_0>0$, $n_0\in \NN^*$ tels que pour tous $0<\epsilon\leq  \epsilon_0$ et $n\geq n_0$:\; $\sup_{x\in X}\;\EE\big[\;exp[\;\epsilon\left(s(x_n,x) \right)\;]\; \big] \leq (1-\epsilon \lambda)^n$.
\end{lemme}
Montrons \`{a} pr\'{e}sent la convergence exponentielle de la marche al\'{e}atoire sur la fronti\`{e}re de Furstenberg. La convergence presque s\^{u}re pour  l'action sur l'espace projectif peut \^{e}tre trouv\'{e}e dans \cite{bougerol}. La preuve utilisait de fa\c{c}on cruciale la d\'{e}composition d'Iwasawa. Le th\'{e}or\`{e}me suivant donne en plus une vitesse exponentielle de convergence et n'utilise pas la d\'{e}composition d'Iwasawa.
\begin{theorem}(Convergence exponentielle de la marche al\'{e}atoire sur la fronti\`{e}re de Furstenberg)\\
On consid\`{e}re  les m\^{e}mes hypoth\`{e}ses qu'au Th\'{e}or\`{e}me \ref{unicite}. On suppose de plus que $\mu$ a un moment exponentiel (D\'{e}finition \ref{defimoment}).  On  note $\delta$ la distance sur la fronti\`{e}re $B$ introduite dans la D\'{e}finition \ref{distfront}.  Alors il existe une variable al\'{e}atoire $Z\in B$ de loi $\nu$ et $\rho\in ]0,1[$ tels que la marche al\'{e}atoire $x_n$ v\'{e}rifie:
\begin{equation}\underset{{x \in B}}{\sup}\; \EE \left( \delta(x_n\cdot x, Z ) \right) \leq \rho^n\label{convvv}\end{equation}
En fait, si $(\rho,V)$ est une repr\'{e}sentation irr\'{e}ductible proximale de $G$ et $\delta$ est la distance de Fubini-Study sur $V$ (voir la D\'{e}finition \ref{distfront}) alors, quitte \`{a} modifier $Z$ et $\rho$, nous avons:
\begin{equation}\underset{{x \in P(V)}}{\sup} \EE \left( \delta(x_n\cdot [x], Z ) \right) \leq \rho^n \label{convv}\end{equation}

\label{convdirection}\end{theorem}
\begin{proof} Par le Lemme \ref{furstprojective}, il suffit de prouver la deuxi\`{e}me partie du th\'{e}or\`{e}me.  Tout d'abord, montrons que

\begin{equation} \underset{n\rightarrow \infty}{\limsup} \frac{1}{n} \log \underset{[x],[y]\in P(V)}{\sup} \EE \left( \delta(x_n \cdot[x],x_n \cdot[y]) \right) < 0 \label{contracdistance}
\end{equation}
En effet,  consid\'{e}rons l'action naturelle de  $G$ sur $X=\{([x],[y])\in P(V)^2; [x]\neq [y]\}$.  Pour tous $g\in X$ et $([x],[y])\in X$,  on pose $s\left( g, ([x],[y]) \right)= \log \frac{\delta(g\cdot x, g\cdot y)}{\delta(x,y)}$. Il est imm\'{e}diat que $s$ est un cocycle sur $G\times X$. L'hypoth\`{e}se de moment du lemme de cocycle est v\'{e}rifi\'{e}e car $\mu$ a un moment exponentiel, donc $\rho(\mu)$ aussi (voir la D\'{e}finition \ref{defimoment}). Avec les notations du m\^{e}me lemme et gr\^{a}ce \`{a} l'expression de la distance $\delta$, on a:
$$l\leq \underset{n\rightarrow \infty}{\lim} \frac{1}{n} \EE (\log  ||\bigwedge^2 \rho(x_n)||) - 2\underset{n\rightarrow \infty}{\lim} \frac{1}{n}  \underset{||u||=1}{\inf} \EE(\log ||\rho(x_n) u ) ||)$$
D'apr\`{e}s la d\'{e}finition de l'exposant de Lyapunov, $ \frac{1}{n}  \EE (\log ||\bigwedge^2 \rho(x_n)||)$ converge vers  $\lambda_1(\rho)+\lambda_2(\rho)$. Gr\^{a}ce \`{a} l'irr\'{e}ductibilit\'{e} de $\rho$, on peut montrer que $\frac{1}{n} \EE(\log  ||\rho(x_n) u ) ||)$ converge vers $\lambda_1(\rho)$ uniform\'{e}ment sur la sph\`{e}re unit\'{e} (voir  \cite[Chapitre III, Corollaire 3.4]{bougerol}). Ainsi $l \leq \lambda_2(\rho)-\lambda_1(\rho)$. Par le Th\'{e}or\`{e}me \ref{interieur}, $l<0$. La conclusion du Lemme \ref{cocycle} et le fait que $\rho \leq 1$ permettent de conclure la preuve de (\ref{contracdistance}).\\

 Finalement, pour tous $k>n$ et tous $[x],[y]\in P(V)$, \'{e}crivons $ \delta(x_n\cdot x, Z) \leq \delta(x_n\cdot [x], x_k \cdot [y]) + \delta(x_k\cdot [y], Z)$. Par ind\'{e}pendance des incr\'{e}ments, on en d\'{e}duit que  $\EE \left( \delta(x_n\cdot x, Z)\right) \leq \underset{[x],[y]\in P(V)}{\sup}\EE \left( \delta(x_n\cdot [x] , x_n \cdot [y])\right) + \EE \left(\delta(x_k\cdot [y], Z)\right)$.  Par l'in\'{e}galit\'{e} (\ref{contracdistance}), la premi\`{e}re quantit\'{e} d\'{e}cro\^{\i}t exponentiellement vite vers $0$ pour $n$ assez grand. Donc il existe $\rho\in ]0,1[$, tel que pour tous $k>n$ assez grands et tous $[x],[y]\in P(V)$:
\begin{equation}\EE \left( \delta(x_n\cdot [x], Z) \right) \leq \rho^n +\EE \left(\delta(x_k\cdot [y],Z )\right)\label{integrer}\end{equation}
D'apr\`{e}s le th\'{e}or\`{e}me de Fubini et le point (i) du Th\'{e}or\`{e}me \ref{unicite},\\$\int {\EE \left(\delta(x_k\cdot [y], Z)\right) \,d\nu([y])} \underset{k \rightarrow + \infty}{\longrightarrow} 0$.  Il suffit alors d'int\'{e}grer (\ref{integrer}) par rapport \`{a} $d\nu([y])$ et de tendre $k$ vers l'infini.  \end{proof}
 
 \section{Equidistribution et ind\'{e}pendance asymptotique sur la fronti\`{e}re de Furstenberg}\label{secasymp}
Le but de ce paragraphe est de d\'{e}montrer le Th\'{e}or\`{e}me
\ref{indasymp}.\\

 Dans cette section, $G$ est un groupe linéaire algébrique réel semi-simple sans facteurs compacts. Pour tout $g\in G$, on note $g=k(g)a(g)u(g)$ le résultat dans la décomposition $G=KA^+K$ (voir le Théorème \ref{theokak}). On note $B$ sa frontière de Furstenberg (voir la Définition \ref{deffront}) et on rappelle que $K$ est un compact maximal de $G$, $M$ le centralisateur de $A$ dans $K$ et que la frontière $B$ s'identifie \`{a} l'espace homogène $K/M$, $M$ étant le centralisateur de $A$ dans $K$.\\

Gorodnik et Oh ont d\'{e}montr\'{e} le r\'{e}sultat suivant:
\begin{theorem}\cite[Th\'{e}or\`{e}me 1.6]{gorodnik}(Ind\'{e}pendance asymptotique pour un comptage sur l'espace sym\'{e}trique)\\
Soient  $d$  une m\'{e}trique
Reimanienne $K$-invariante sur l'espace sym\'{e}trique
$K\backslash G$ et
 $\nu$ la mesure de Haar sur $K$.  Pour tout $T\in {\RR^{+ *}}$, on note $A_T^+=\{a\in A^+; d(K, K a) < T\}$ et $G_T=\{g\in G;  a(g)\in A_T^+\}$.
   On consid\`{e}re un r\'{e}seau $\Gamma$ de $G$.\\
Alors pour tous $\Omega_1 \subset K$, $\Omega_2 \subset K$ dont la fronti\`{e}re est de $\nu$-mesure  nulle:
$$\textrm{Card}\{g\in G_T\cap \Gamma;\; k(g)M \subset \Omega_1 M; \;Mu(g) \subset M \Omega_2\} \underset{T \rightarrow + \infty}{\sim}\nu(\Omega_1 M) \nu( M \Omega_2) \frac{\textrm{Vol}(G_T)}{\textrm{Vol}(G/\Gamma)}$$
\label{gorodoh} \end{theorem}

Outre l'\'{e}quidistribution vers la mesure de Haar sur la
fronti\`{e}re de Furstenberg $B$, ce r\'{e}sultat  montre
que les deux composantes $k(g)M\in B$ et $Mu(g)\in B^*$ d'un
\'{e}l\'{e}ment $g=k(g)a(g)u(g)\in \Gamma$ pris uniform\'{e}ment
au hasard dans une grande boule de l'espace sym\'{e}trique,
peuvent \^{e}tre choisies ind\'{e}pendamment l'une de l'autre
suivant la mesure de Haar. Le Th\'{e}or\`{e}me \ref{indasymp}
propose un \'{e}nonc\'{e} \'{e}quivalent pour un comptage
probabiliste. La mesure de Haar est remplac\'{e}e par les mesures
stationnaires sur $B$ et $B^*$ (Th\'{e}or\`{e}me \ref{unicite}). De plus, les vitesses de convergences obtenues sont exponentielles.

\paragraph{Points cl\'{e}s de la preuve} La preuve du Th\'{e}or\`{e}me \ref{indasymp}
 repose sur deux points cruciaux: le Lemme \ref{lemmeind} ci-dessous, qui est un  lemme g\'{e}n\'{e}ral d'ind\'{e}pendance asymptotique, et
 le Th\'{e}or\`{e}me \ref{convkak} qui montre la convergence exponentielle des parties $K$ de la d\'{e}composition $KA^+K$ de la marche aléatoire. 

\begin{propdef}(Convergence exponentielle de variables al\'{e}atoires)
Nous dirons qu'une suite $(X_n)_{n\in \NN}$ de variables al\'{e}atoires \`{a} valeurs dans un espace m\'{e}trique complet $(X,d)$  converge exponentiellement vite  vers une variable al\'{e}atoire $X_{\infty}$ si $\underset{n\rightarrow \infty}{\limsup}\frac{1}{n}\log \EE \left( d(X_n, X_{\infty})\right) <0$. Par compl\'{e}tude de l'espace, ceci est \'{e}quivalent \`{a} supposer que $\underset{n\rightarrow \infty}{\limsup}\frac{1}{n}\log \EE \left( d(X_n, X_{n+1})\right) <0$.  Remarquer que par Borel-Cantelli,  une telle suite converge presque s\^{u}rement. \label{convexpdef}\end{propdef}

 \begin{lemme}(Ind\'{e}pendance asymptotique)\\
 Soit $(\Gamma,\star)$ un mono\"{\i}de muni d'une mesure de probabilit\'{e} $\mu$. On d\'{e}finit sur un espace probabilis\'{e} $(\Omega, \Ff,\PP)$ une suite $(g_i)_{i\geq 1}$  de variables al\'{e}atoires ind\'{e}pendantes de m\^{e}me loi $\mu$.  Soient  $Y_1$ et $Y_2$ deux espaces m\'{e}triques complets, $\alpha$ et $\beta$ deux fonctions mesurables d\'{e}finies sur $\Gamma$ \`{a} valeurs respectivement  dans $Y_1$ et $Y_2$. On pose $F_n=\alpha(g_1\star \cdots \star g_n )$, $G_n=\beta(g_1 \star \cdots \star g_n)$ et $H_n=\beta(g_n \star \cdots \star g_1)$.\\

\emph{Version sans vitesse:} Si  $F_n$ et $H_n$ convergent presque s\^{u}rement respectivement dans $Y_1$ et $Y_2$ vers des variables al\'{e}atoires de lois respectives $\nu_1$ et $\nu_2$, alors $F_n$ et $G_n$ deviennent asymptotiquement ind\'{e}pendants, dans le sens suivant: pour tous ouverts $O_1$ et $O_2$ respectivement dans  $Y_1$ et $Y_2$ tels que $\nu_1(\partial O_1)=0$ et $\nu_2(\partial O_2)=0$, alors:

 $$\PP(F_n \in O_1; G_n \in O_2)\underset{n \rightarrow \infty}{\longrightarrow} \nu_1(O_1) \nu_2(O_2)$$

\emph{Version avec vitesse:} Si $F_n$ et $H_n$ convergent exponentiellement vite (voir la D\'{e}finition \ref{convexpdef}) alors $F_n$ et $G_n$ deviennent asymptotiquement ind\'{e}pendants avec vitesse exponentielle dans le sens suivant: il existe $\rho\in ]0,1[$ tel que pour toute fonction r\'{e}elle lipschitzienne $\phi$ d\'{e}finie sur $Y_1\times Y_2$,
  $$\big| \int_{\Gamma} {\phi \left( \alpha(\gamma) , \beta(\gamma) \right)\; \textrm{d}\mu^n(\gamma)} - \int_{Y_1 \times Y_2}{\phi(x,y)\; \textrm{d}( {\nu_1\otimes\nu_2}) (x,y) } \big| \leq \textrm{Lip($\phi$)}\; \rho^n$$\label{lll}\label{lemmeind}\end{lemme}

\begin{proof}
Nous d\'{e}montrons la version avec vitesse. Celle sans vitesse peut \^{e}tre d\'{e}montr\'{e}e de mani\`{e}re analogue.  Soit $Z_1$ la limite presque s\^{u}re  de la suite des variables al\'{e}atoires  $F_n$. On d\'{e}finit sur le m\^{e}me espace probabilis\'{e} une copie ind\'{e}pendante $(g'_i)_{i\geq 1}$ de $(g_i)_{i\geq 1}$.  L'\'{e}quivalence dans la D\'{e}finition/Proposition \ref{convexpdef} permet de d\'{e}finir une variable al\'{e}atoire $Z_2$, ind\'{e}pendante de $Z_1$, telle que la suite $\beta(g'_n\star \cdots \star g'_1)$ converge de fa\c{c}on exponentielle vers $Z_2$. En particulier, $Z_1$ et $Z_2$ sont ind\'{e}pendantes. Soit $\phi$ une fonction lipschitzienne sur $Y_1\times Y_2$. On \'{e}crit:
$$E\left(\phi(F_n,G_n) \right) -\EE\left(\phi(Z_1,Z_2) \right)= \Delta_1 + \Delta_2 + \Delta_3+\Delta_4$$
o\`{u}
\begin{eqnarray}
 \Delta_1&=&\EE\left(\phi(F_n,G_n) \right) - \EE\left(\phi(F_{\frac{n}{2}},G_n) \right)\nonumber\\ \Delta_2&=&\EE\left(\phi(F_{\frac{n}{2}},G_n) \right)- \EE\left(\phi(F_{\frac{n}{2}}, \beta(g_{\frac{n}{2}+1}\star \cdots \star g_n) ) \right) \nonumber\\
\Delta_3&=& \EE\left(\phi(F_{\frac{n}{2}}, \beta(g'_{\frac{n}{2}+1}\star \cdots \star g'_n)) \right) - \EE\left(\phi(Z_1,\beta(g'_{\frac{n}{2}+1} \star \cdots \star g'_n))\right)\nonumber\\
\Delta_4&=&\EE\left(\phi(Z_1,\beta(g'_{\frac{n}{2}+1} \star \cdots \star g'_n) )\right) - \EE\left(\phi(Z_1,Z_2)\right)  \nonumber
\end{eqnarray}

\begin{itemize}
\item
   Dans $\Delta_3$, nous avons remplac\'{e} $\beta(g_{\frac{n}{2}}\star \cdots \star g_n)$  par $\beta(g'_{\frac{n}{2}+1} \star \cdots \star g'_n)$ car, d'une part ces variables ont m\^{e}me loi et d'autre part les variables $F_{\frac{n}{2}}$ et $\beta(g_{\frac{n}{2}+1}\star \cdots \star g_n)$ qui apparaissent dans  $\Delta_2$ sont ind\'{e}pendantes. Cette \'{e}tape est simple mais cruciale pour la preuve.

\item  D'apr\`{e}s l'hypoth\`{e}se, les quantit\'{e}s $|\Delta_1|$ et $|\Delta_3|$ convergent exponentiellement vite vers z\'{e}ro.

\item  La deuxi\`{e}me \'{e}tape importante de la preuve est l'estim\'{e}e $\Delta_2$. Nous affirmons qu'il existe $\rho_2\in ]0,1[$ tel que pour tout entier $n$ assez grand: $|\Delta_2|\leq \rho^n$.
En effet, pour tout entier $n\in \NN^*$, les $n$-uplets $(g_1,\cdots, g_n)$ et $(g_n,\cdots, g_1)$ ont m\^{e}me loi. En notant $\delta$ la m\'{e}trique sur $Y_2$, on en d\'{e}duit que:
\begin{eqnarray}\EE \Big( \delta\left(\beta(g_1 \star \cdots \star g_n ),\beta(g_{\frac{n}{2}+1}\star \cdots  \star g_n ) \right)\Big) &=&
\EE  \Big(  \delta \left(\beta(g_n \star \cdots  \star g_1 ),\beta(g_{\frac{n}{2}}\star \cdots \star g_1 ) \right)\Big)\nonumber\\
&=&\EE \left( \delta(H_n ,H_{\frac{n}{2}}) \right)\nonumber\end{eqnarray}
D'apr\`{e}s l'hypoth\`{e}se,  $H_n$ converge exponentiellement vite dans $Y_2$. Donc la quantit\'{e} ci-dessus d\'{e}cro\^{\i}t exponentiellement vite vers $0$. Donc $|\Delta_2|$ aussi.

\item Finalement, $|\Delta_4|\preceq \textrm{Lip($\phi$)}\; \rho_2^n$. En effet, l'ind\'{e}pendance entre $Z_1$ et la suite des $g'_i$ d'une part et l'\'{e}galit\'{e} en loi des $n$-uplets  $(g'_1,\cdots, g'_n)$ et $(g'_n,\cdots, g'_1)$ d'autre part permettent d'\'{e}crire:
$ \Delta_4=\EE\left(\phi(Z_1, \beta(g'_{\frac{n}{2}} \star \cdots \star g'_1) )\right) - \EE\left(\phi(Z_1,Z_2)\right)$. On conclut gr\^{a}ce \`{a} la convergence exponentielle de\\  $\beta(g'_n\star \cdots \star g'_1)$.
\end{itemize}
  \end{proof}

\begin{remarque}
\begin{enumerate}
\item Ce lemme implique que si  $\alpha(g_1\star \cdots \star g_n)$ converge presque s\^{u}rement alors $\alpha(g_n\star \cdots \star g_1)$ ne peut pas converger  presque s\^{u}rement, sauf si les deux limites sont des constantes.
\item En particulier, quand $G$ est ab\'{e}lien, une fonction mesurable   $\alpha(x_n)$ de $x_n$ ne peut converger presque s\^{u}rement que vers une constante. Ce fait est \`{a} comparer avec la trivialit\'{e} de la fronti\`{e}re de Poisson de $G$.
\item Nous r\'{e}f\'{e}rons aux travaux de Tutubalin \cite{tutubalin} et Vircer \cite{vircer} repris par Guivarc'h \cite{Zariskiclosure}, o\`{u} l'id\'{e}e de couper le temps en deux est utilis\'{e}e. \end{enumerate}
\label{bonnerem}\end{remarque}
Le Th\'{e}or\`{e}me suivant montre la convergence exponentielle de la composante gauche suivant  $K$ dans la d\'{e}composition de la marche al\'{e}atoire $x_n=g_1\cdots g_n$ dans le produit $KA^+K$. Nous proposons une preuve  plus directe que celle donn\'{e}e par l'auteur dans \cite{aoun} n'utilisant pas la d\'{e}composition d'Iwasawa. Par contre, cette preuve n'est valide que sur $\RR$ contrairement au travail cit\'{e} car elle utilise de fa\c{c}on cruciale les projections orthogonaux des espaces euclidiens. Elle est inspir\'{e}e de la preuve de Goldsheid-Margulis du Th\'{e}or\`{e}me d'Oseledets  \cite[Th\'{e}or\`{e}me 1.2]{Margulis}.

\begin{theorem}(Convergence exponentielle dans la d\'{e}composition $KA^+K$ sur la fronti\`{e}re de Furstenberg)\\
 On consid\`{e}re un sous-groupe Zariski dense $\Gamma$ de $G$, $\mu$ une mesure de probabilit\'{e} adapt\'{e}e (D\'{e}finition \ref{defgen}) ayant un moment
exponentiel (D\'{e}finition \ref{defimoment}) et une marche al\'{e}atoire
$x_n = g_1\cdots g_n$ sur $\Gamma$, avec les $g_i$ ind\'{e}pendants et
de m\^{e}me loi $\mu$.\\
Alors la composante $k(x_n) M$ de la marche
al\'{e}atoire $x_n$ converge presque s\^{u}rement dans la fronti\`{e}re de
Furstenberg $B$, avec vitesse exponentielle, vers une variable
al\'{e}atoire $Z_1$ de loi l'unique mesure $\mu$-invariante
$\nu_1$ sur la fronti\`{e}re. Plus pr\'{e}cis\'{e}ment, si $\delta$
est la distance sur la fronti\`{e}re de Furstenberg $B$ introduite
dans la D\'{e}finition \ref{distfront}, il existe $\rho_1\in ]0,1[$
tel que:
\begin{equation}\int {\delta\left(k(g_1\cdots g_n )M,Z_1\right)  \,d\mu(g_1)\cdots d\mu(g_n) )} \leq \rho_1^n\label{conv1}\end{equation}
De fa\c{c}on similaire, la composante $M u(y_n)$ de la marche al\'{e}atoire \`{a} gauche $y_n=g_n\cdots g_1$ converge de fa\c{c}on exponentielle vers une variable al\'{e}atoire $Z_2$ sur $B^*=M\backslash K$:
\begin{equation}\int {\delta\left(M u(g_n\cdots g_1),Z_2\right) \,d\mu(g_1)\cdots d\mu(g_n)}\leq \rho_2^n\label{conv2}\end{equation}

\label{convkak}\end{theorem}
\begin{remarque} Comme pour tout $n\in \NN^*$, les $n$-uplets $(g_1,\cdots, g_n)$ et $(g_n,\cdots, g_1)$ ont m\^{e}me loi, alors la variable $Mu(x_n)=Mu(g_1\cdots g_n)$ converge en loi. La remarque \ref{bonnerem} implique qu'elle ne peut pas converger presque s\^{u}rement, la mesure stationnaire $\nu$ \'{e}tant non atomique (voir \cite{bougerol}). \label{hehe}\end{remarque}
\begin{proof}
Le Lemme \ref{furstprojective} permet de ramener l'\'{e}tude \`{a} celle d'une repr\'{e}sentation irr\'{e}ductible et proximale.   Soit  donc $(\rho,V)$ une telle repr\'{e}sentation de $G$, $F=(e_1,\cdots, e_d)$ la base donn\'{e}e par le Corollaire \ref{mostowcons} et $\langle \cdot, \cdot \rangle$ le produit scalaire donn\'{e} par le m\^{e}me corollaire. On rappelle que $\rho(K)$ agit par matrices orthogonales et $\rho(A)$ est constitu\'{e} de matrices diagonales.  Il suffit de d\'{e}montrer qu'il existe des variables al\'{e}atoires $Z_1\in P(V)$ et $Z_2\in P(V^*)$ et un r\'{e}el $\rho\in ]0,1[$ tel que:
\begin{equation}\EE \left( \delta(k(x_n)\cdot[e_1],Z_1)^\epsilon \right) \leq \rho ^n\;\;\;\;;\;\;\;\;\EE \left( \delta(u(y_n)^{-1}\cdot[e_1^*],Z_2)^\epsilon \right) \leq \rho ^n\label{conv3}\end{equation}
Pour tout $x\in V$, notons $Q_x$ la projection orthogonale sur la droite $\RR x$. Pour simplifier les notations, nous \'{e}crivons,  pour tout $n\in \NN^*$ et tout $j\in \{1,\cdots, d\}$,   $Q_{e_j}(n)$ au lieu de $Q_{ k(x_n)\cdot e_j}$. Nous allons montrer que pour tout $\epsilon>0$ assez petit:
\begin{equation}
\displaystyle \limsup_{n\rightarrow \infty}\big[\EE(||Q_{e_1}(n+1) - Q_{e_1}(n)||^\epsilon)\big]^{\frac{1}{n}}< 1\label{ch2yara}\end{equation}
La compl\'{e}tude de l'espace projectif et l'in\'{e}galit\'{e} $||Q_x - Q_y||\geq \frac{1}{2}\delta([x],[y])$ vraie pour tout $x,y\in V$ permettent d'obtenir la premi\`{e}re in\'{e}galit\'{e} de (\ref{conv3})  \`{a} partir de (\ref{ch2yara}). La deuxi\`{e}me se prouve de fa\c{c}on similaire en regardant l'action de $G$ sur le dual de $V$.
Comme $\sum_{j=1}^d {Q_{e_j}(n)}$ est l'op\'{e}rateur identit\'{e} pour tout entier $n$, alors le calcul ci-dessous  ram\`{e}ne l'\'{e}tude de l'esp\'{e}rance de l'estim\'{e}e (\ref{ch2yara}) \`{a} celle de deux esp\'{e}rances:
\begin{eqnarray}
||Q_{e_1}(n+1) - Q_{e_1}(n)|| & \leq & \sum_{j\neq 1} { ||Q_{e_1}(n+1)\circ Q_{e_j}(n)||+ || Q_{e_j}(n+1) \circ Q_{e_1}(n)|| }\nonumber\\
& \leq& \sum_{j\neq 1} { ||Q_{e_1}(n+1) \circ Q_{e_j}(n)||+ || Q_{e_1}(n) \circ Q_{e_j}(n+1) || } \label{bih}\end{eqnarray}
(\ref{bih}): si $p$ et $q$ sont deux projections orthogonaux de $V$ alors $||p \circ q||=||q \circ p||$. \\
Soient $j\neq 1$,  $x\in \RR k(x_n)\cdot e_j$ de norme $1$ et  $y_n=Q_{e_1}(n+1)(x)$. Pour simplifier les notations, notons $x_n=g_1\cdots g_n$ au lieu de $\rho(x_n)$ et $\rho\left( a(x_n) \right)=diag\left(a_1(n),\cdots, a_d(n)\right)$ la matrice  diagonale  \'{e}crite dans la base $F$.  Evaluons  $||x_{n+1}^t x||$ de deux fa\c{c}ons diff\'{e}rentes. D'une part,
 \begin{equation}||x_{n+1}^t  x||=||g_{n+1}^t \cdots g_1^t  x||\leq ||g_{n+1}|| \smallskip||x_n^tx||= ||g_{n+1}||\smallskip a_j(n)\label{ch21}\end{equation}
D'autre part, $\langle x_{n+1}^t  (x-y_n), x_{n+1}^t  y_n\rangle = 0$ car
$x_{n+1} x _{n+1}^ t \cdot y_n \in k(x_{n+1})\cdot [e_1]$. D'o\`{u},
\begin{equation}\label{ch22}||x_{n+1}^t   x||=\sqrt{||x_{n+1}^t   y_n||^2+||x_{n+1}^t   (x-y_n)||^2} \geq
||x_{n+1}^t  y_n||=a_1(n+1)||y_n||\end{equation}
En combinant (\ref{ch21}) et (\ref{ch22}) et l'\'{e}galit\'{e} $||a_1(n)||=||x_n||$, nous obtenons:
\begin{equation}||Q_{e_1}(n+1)Q_{e_j}(n)|| \leq ||x_{n+1}|| \;\frac{a_j(n)}{a_1(n+1)} \leq ||x_{n+1}||  ||x_{n+1}^{-1}|| \;\frac{a_j(n)}{a_1(n)} \label{ch2badama3}\end{equation}
En appliquant la Proposition \ref{propratio} suivante, l'hypoth\`{e}se de moment exponentiel puis l'in\'{e}galit\'{e} de Cauchy-Schwartz \`{a} (\ref{ch2badama3}), nous obtenons la d\'{e}croissance exponentielle de la quantit\'{e} $\EE \left( ||Q_{e_1}(n+1)Q_{e_j}(n)||^\epsilon \right)$ pour $\epsilon$ assez petit. \\
Pour majorer l'esp\'{e}rance de $||Q_{e_1}(n)Q_{e_j}(n+1)||$, il suffit d'appliquer le m\^{e}me raisonnement \`{a} $||x_n^t x|| \leq ||g_{n+1}^{-1}||\;||x_{n+1}^t \cdot x||$ avec $x\in \RR k(x_{n+1})\cdot e_j$. \\
La premi\`{e}re in\'{e}galit\'{e} de (\ref{conv3}) est d\'{e}montr\'{e}e.    \end{proof}
La proposition suivante est une  version en esp\'{e}rance du Th\'{e}or\`{e}me \ref{interieur}.
\begin{proposition}\label{propratio}
On consid\`{e}re les m\^{e}mes hypoth\`{e}ses et notations du Th\'{e}or\`{e}me \ref{indasymp}. Soit $(\rho,V)$ une repr\'{e}sentation irr\'{e}ductible, rationnelle et proximale de $G$. On note $\rho\left(a(x_n)\right) = diag\left( a_1(n),\cdots, a_d(n)\right)$ dans la base $F=(e_1,\cdots, e_d)$ du Corollaire \ref{mostowcons}. Alors, pour tout $j\neq 1$:  $$ \underset{n \rightarrow +\infty}{\limsup} \frac{1}{n} \log \EE \left(\frac{a_j(n)}{a_1(n)}\right)< 0$$\end{proposition}
\begin{proof}
D'apr\`{e}s le Corollaire \ref{mostowcons}, l'in\'{e}galit\'{e} suivante $\frac{a_j(n)}{a_1(n)} \leq  \frac{||\bigwedge^2 x_n||}{||x_n||^2}$ est v\'{e}rifi\'{e}e pour tout $j\neq 1$. Donc il existe $C>0$ tel que pour tout $\epsilon>0$,
\begin{equation}\EE \left( \Big[\frac{a_j(n)}{a_1(n)}\Big]^\epsilon \right) \leq C\underset{i\in \{1,\cdots, d\}}{\sup} \EE \left( \Big[\frac{||\bigwedge^2 x_n e_i ||}{||x_n||^2}\Big]^\epsilon \right)\label{interm}\end{equation}
On consid\`{e}re \`{a} pr\'{e}sent le cocycle $s$ d\'{e}fini sur $G \times (P(V)\times P(V) )$ par:
$s\left(g,([x],[y]]\right)= \log \frac{||\bigwedge^2 g\cdot x||\;||y||^2}{||g\cdot y||^2\;||x|| }$. L'hypoth\`{e}se du Lemme \ref{cocycle} de cocycle est v\'{e}rifi\'{e}e car $\rho(\mu)$ a un moment exponentiel (D\'{e}finition/Proposition \ref{defimoment}). Avec les notations du Lemme \ref{cocycle} de cocyle, $l\leq \lambda_2(\rho) - \lambda_1(\rho)$. Par le Th\'{e}or\`{e}me \ref{interieur}, $l<0$. La conclusion du lemme permet d'obtenir la d\'{e}croissance exponentielle de (\ref{interm}) pour $\epsilon>0$ assez petit. Comme $a_j(n)/a_1(n) < 1$, alors le r\'{e}sultat est vrai pour tout $\epsilon>0$.   \end{proof}

Le Th\'{e}or\`{e}me \ref{indasymp} d\'{e}coule imm\'{e}diatement du Lemme \ref{lemmeind} et du Th\'{e}or\`{e}me \ref{convkak} en posant $Y_1=B$, $Y_2=B^*$, $\alpha (g)=k(g)M$ et $\beta(g)=M u(g)$ pour tout $g\in G$.

\section{Une preuve simple de la positivit\'{e} de la dimension de Hausdorff de la mesure stationnaire}\label{sechausdorff}

Le but de ce paragraphe est de donner une nouvelle preuve du Th\'{e}or\`{e}me \ref{hausdorff}. On rappelle que si $(X,\delta)$  est un espace m\'{e}trique et $\nu$ une mesure de probabilit\'{e} bor\'{e}lienne sur $X$ alors la dimension de Hausdorff de $\nu$ est d\'{e}finie par  $\HD(\nu)=\inf \{\HD(A); \nu(A)=1\}$. Ici $\HD(A)$ est la dimension de Hausdorff classique du bor\'{e}lien $A$.  Nous renvoyons \`{a} \cite{mattila} et \cite{ledrappier} pour les d\'{e}tails.

Le Th\'{e}or\`{e}me \ref{hausdorff} d\'{e}coulera du r\'{e}sultat suivant:
\begin{theorem}
Soit $G$ un groupe lin\'{e}aire alg\'{e}brique r\'{e}el
semi-simple et sans facteurs compacts et $(\rho,V)$ une
repr\'{e}sentation irr\'{e}ductible proximale de $G$. On consid\`{e}re
un sous-groupe Zariski dense $\Gamma$ de $G$,  muni d'une mesure
de probabilit\'{e} adapt\'{e}e $\mu$
 (D\'{e}finition \ref{defgen}) et ayant un moment exponentiel (D\'{e}finition \ref{defimoment}). Soit $\nu$ l'unique mesure de
 probabilit\'{e} $\mu$-invariante sur l'espace projectif $P(V)$ (Th\'{e}or\`{e}me \ref{unicite}).  Alors il existe $C>0$, $\alpha>0$  et $\epsilon_0>0$ tels que pour tout $\epsilon \in [0,\epsilon_0]$:
\begin{equation}\sup \{\nu \{ [x]\in P(V); \delta(x, H) \leq \epsilon \}; \textrm{$H$ hyperplan de $V$}\} \leq C\epsilon^\alpha\label{aprouver}\end{equation}En particulier, $\HD(\nu)\geq \alpha$.
Ici,  $\delta$ est la distance de Fubini-Study sur l'espace projectif $P(V)$ (D\'{e}finition \ref{distfront}).

\label{hausdorff1}\end{theorem}

\begin{proof}[Preuve du Th\'{e}or\`{e}me \ref{hausdorff} modulo le Th\'{e}or\`{e}me \ref{hausdorff1}:]
 D'apr\`{e}s le Lemme \ref{furstprojective}, la fronti\`{e}re de Furstenberg est plong\'{e}e dans $\prod_{i=1}^r P(V_i)$
 avec les $(\rho_i,V_i)$ des repr\'{e}sentations irr\'{e}ductibles proximales de $G$.
 Il est clair que la projection de la mesure stationnaire $\nu$  sur chaque espace projectif $P(V_i)$ est aussi stationnaire.
 D'apr\`{e}s le Th\'{e}or\`{e}me \ref{unicite}, elle est unique. Ainsi le th\'{e}or\`{e}me pr\'{e}c\'{e}dent implique le
 Th\'{e}or\`{e}me \ref{hausdorff}.  \end{proof}

\paragraph{Idée de la preuve}
Dans sa preuve, Guivarc'h utilise des th\'{e}or\`{e}mes limites sur la d\'{e}composition d'Iwasawa
pour d\'{e}montrer  une assertion plus forte que
 (\ref{aprouver})\footnote{\`{a} savoir $\int {\frac{1}{\delta([x],H)^\alpha}\,d\nu([x])} < \infty$ uniform\'{e}ment en les hyperplans $H$.
 En termes de la th\'{e}orie g\'{e}om\'{e}trique de la mesure, cela implique que la $\alpha$-\'{e}nergie de $\nu$ est finie.}.
 Nous proposons une d\'{e}marche qui n'utilise  que la convergence exponentielle de la marche al\'{e}atoire $x_n$ en direction
 (Th\'{e}or\`{e}me \ref{convdirection}). L'id\'{e}e principale est d'approximer, pour tout $\omega\in \Omega$ appartenant \`{a}
 un ensemble de grande probabilit\'{e}, la distance  de $Z(\omega)$ \`{a} un hyperplan donn\'{e} $H$  par un rapport de normes.
  Il s'agira du quotient $\frac{||\rho\left(x_n(\omega)\right)^t   v(\omega)  ||}{||\rho\left(x_n(\omega)\right)||}$, o\`{u} $n$ est un
  entier assez grand,  $v(\omega)$  est un certain vecteur de $V$ variant al\'{e}atoirement parmi un ensemble fini de vecteurs d\'{e}terministes.
   Nous renvoyons aux in\'{e}galit\'{e}s (\ref{live}) et l'estim\'{e}e (\ref{te3ebt}).
   Le contr\^{o}le des rapports de normes sera fait gr\^{a}ce \`{a} la Proposition \ref{GDV},  version modifi\'{e}e du lemme de cocycle (Lemme \ref{cocycle}).
\newline
\begin{proof}[Preuve du Th\'{e}or\`{e}me \ref{hausdorff1}:]
Soient $\langle \cdot, \cdot \rangle$ le produit scalaire
 et $F=(e_1,\cdots,
e_d)$ la base orthonorm\'{e}e donn\'{e}s par le Corollaire
\ref{mostowcons}. On note $Z$ la variable al\'{e}atoire donn\'{e}e
par le Th\'{e}or\`{e}me \ref{convdirection}. L'in\'{e}galit\'{e}
(\ref{aprouver}) est \'{e}quivalente \`{a} prouver l'assertion
suivante:  pour tout $\rho\in ]0,1[$, il existe $\rho'\in ]0,1[$
et $n_0\in \NN$ tels que pour tous $n\geq n_0$ on a:
\begin{equation} \sup\{\PP \left( \delta(Z,H) \leq \rho^n \right); \textrm{$H$ hyperplan de $V$}\} \leq \rho'^n \label{bahebbakpapa}\end{equation}
Soit $H$ un hyperplan de $V$ et $y\in V$ un vecteur unitaire orthogonal \`{a} $H$. Alors $\delta(a,H)=\frac{|\langle a,y\rangle|}{||a||}$ pour tout $a\in V$. Comme $||\rho(x_n)^t y||= \sum_{i=1}^d {|\langle x_n \cdot e_i, y\rangle|^2}$, alors pour $\PP$-presque tout $\omega \in \Omega$,  il existe $i(\omega)\in \{1,\cdots, d\}$ tel que:
\begin{equation}\delta\left(x_n(\omega)\cdot e_{i(\omega)},H\right)\geq \frac{1}{\sqrt{d}} \frac{||\rho(x_n)^t\; y ||}{||\rho(x_n)||}\label{live}\end{equation}
Gr\^{a}ce au Th\'{e}or\`{e}me \ref{convdirection}, on en d\'{e}duit l'existence de $\rho_1,\rho_2\in ]0,1[$ tels que:
\begin{eqnarray}
\PP\left( \delta(Z,H) \leq \rho^n \right) &=& \sum_{i=1}^d {\PP\left( \delta(Z,H) \leq \rho^n; \;\mathds{1}_{i(\omega)=i}\right)}\nonumber\\
&\leq& \sum_{i=1}^d {\PP\left( \delta(x_n \cdot e_i, H)\leq \rho_1^n + \rho^n; \;\mathds{1}_{i(\omega)=i} \right)} + \rho_2^n \nonumber\\
& \leq & \sum_{i=1}^d {\PP\left( \frac{||\rho(x_n)^ty ||}{||\rho(x_n)||}\leq \sqrt{d}(\rho_1^n+\rho^n) \right)} + \rho_2^n \label{te3ebt}\end{eqnarray}
La proposition suivante appliqu\'{e}e \`{a} la mesure de probabilit\'{e} $\rho(\mu)^t$, loi de $\rho(g_1)^t$, permet de montrer que la quantit\'{e} pr\'{e}c\'{e}dente d\'{e}cro\^{\i}t exponentiellement vite. L'in\'{e}galit\'{e}  (\ref{bahebbakpapa}) est donc d\'{e}montr\'{e}e.   \end{proof}
\begin{proposition}(Grandes d\'{e}viations de rapports)\\
Soit  $\mu$ une mesure de probabilit\'{e} sur $\operatorname{SL}_d(\RR)$ dont le support engendre un groupe irr\'{e}ductible et d'adh\'{e}rence Zariski connexe. Alors pour tout $t\in ]0,1[$, $\underset{n\rightarrow +\infty}{\limsup}\, {\frac{1}{n} \log \underset{x,y\in V; ||x||=||y||=1}{\sup} \PP \left( \frac{||x_n x||}{||x_n y ||}\leq t^n \right)} < 0$. \label{GDV}\end{proposition}
La proposition peut \^{e}tre d\'{e}montr\'{e}e gr\^{a}ce aux travaux de Le Page \cite{page} expliqu\'{e}s dans la Remarque \ref{LGN}. Une m\'{e}thode plus directe est le recours \`{a} une deuxi\`{e}me version du Lemme \ref{cocycle}.
\begin{lemme}(Lemme de cocycle: deuxi\`{e}me version)\cite[Lemme 4.12]{aoun} On consid\`{e}re la m\^{e}me situation que le Lemme \ref{cocycle}. Si $l=0$, alors   pour tous $\gamma>0$, il existe $\epsilon(\gamma)>0$, $n(\gamma)\in \NN^*$ tels que pour tous $0<\epsilon<\epsilon(\gamma)$ et  $n>n(\gamma)$, $Sup_{x\in X}\;\EE\big[\;exp[\;\epsilon\left(s(x_n,x) \right)\;]\; \big] \leq (1+\epsilon \gamma)^n$. \label{cocycle2}\end{lemme}
\begin{proof}[Preuve de la Proposition \ref{GDV}: ]
On consid\`{e}re l'action naturelle de $G$ sur $X=P(V) \times P(V)$ et $s$ le cocycle sur $G\times X$ d\'{e}fini par
$s\left(g,([x],[y])\right)= \log \frac{||g x||\;||y||}{||g y||\;||x||}$. L'hypoth\`{e}se d'irr\'{e}ductibilit\'{e} permet de montrer que $\frac{1}{n} \EE(\log||x_n u||)$ converge vers l'exposant de Lyapunov $\lambda_1$ uniform\'{e}ment en les vecteurs $u$ de la sph\`{e}re unit\'{e} \cite[Chapitre III, Corollaire 3.4]{bougerol}. Ainsi $l=0$ et on peut alors appliquer  le Lemme \ref{cocycle2}. On conclut gr\^{a}ce \`{a} l'in\'{e}galit\'{e} de Markov.    \end{proof}

\paragraph{Probl\`{e}me ouvert:}
 Soient $d,k\in \NN^*$. En combinant l'alternative de Tits forte de Breuillard, plus pr\'{e}cis\'{e}ment \cite[Corollaire 1.6]{strongtits},   et la minoration de l'exposant de Lyapunov d'une mesure de probabilit\'{e} par une fonction de son rayon spectral \cite[Th\'{e}or\`{e}me 1.19]{furman}, on d\'{e}duit  qu'il existe une borne uniforme  de l'exposant de Lyapunov. Plus pr\'{e}cis\'{e}ment, il existe $C=C(d,k)$ tel que $\lambda_1>C$ pour toute mesure de probabilit\'{e}  uniforme $\mu$ sur un ensemble de cardinal $k$ et engendrant un groupe $G_\mu$ non moyennable. Peut-on obtenir un r\'{e}sultat analogue sur la diff\'{e}rence des exposants de Lyapunov si on impose par exemple que $G_\mu$ est Zariski dense? Et par la suite,  existe-t-il une borne uniforme de la dimension de la mesure stationnaire?

\section{Une version probabiliste de l'alternative de Tits}
\label{secping}
Le but de cette partie est de d\'{e}montrer le Th\'{e}or\`{e}me \ref{titsproba}. \\
Dans toute cette section, $G$ est un groupe lin\'{e}aire
alg\'{e}brique r\'{e}el semi-simple et sans facteurs compacts.\\

Apr\`{e}s son utilisation cruciale par Tits dans la preuve de l'alternative qui porte son nom \cite{tits},  la  m\'{e}thode de Ping-Pong est devenue classique pour d\'{e}montrer qu'un groupe est libre.   En voici une des versions possibles:
\begin{lemme}(Lemme de Ping-Pong)
Soit $G$ un groupe agissant sur un ensemble $X$. On consid\`{e}re deux \'{e}l\'{e}ments $g$ et $h$ de $G$. On suppose qu'il existe
$\emptyset \varsubsetneqq V_g \subset X$ (resp. $V_{g^{-1}}$, $V_h$ et $V_{h^{-1}}$) et $H_g  \varsubsetneqq X$ (resp. $H_{g^{-1}}$, $H_h$ et $H_{h^{-1}}$) tels que:
\begin{enumerate}
\item $g \cdot (X \setminus H_g) \subset V_g$. Idem en rempla\c{c}ant $(V_g,H_g)$ respectivement par $(V_{g^{-1}},H_{g^{-1}})$, $(V_h, H_h)$, $(V_{h^{-1}}, H_{h^{-1}})$.
\item $V_g \cap  H_g= \emptyset $. Idem en rempla\c{c}ant $(V_g,H_g)$ respectivement par $(V_{g^{-1}},H_{g^{-1}})$, $(V_h, H_h)$, $(V_{h^{-1}}, H_{h^{-1}})$.
\item $V_{g^{\pm 1}} \cap  H_{h^{\pm 1}}= \emptyset $ et $V_{h^{\pm 1}} \cap  H_{g^{\pm 1}}= \emptyset $.
\item $V_g \cup H_g \neq X$, $V_{g^{-1}}\cup H_{g^{-1}}\neq X$ et $V_g \cup H_{h^{\pm 1}}\neq X$. Idem en permutant les r\^{o}les de $g$ et $h$. \end{enumerate}
Alors le groupe $\langle g, h \rangle $ engendr\'{e} par $g$ et $h$   est libre.\\

\noindent Quand la situation pr\'{e}c\'{e}dente se pr\'{e}sente, on dit que $g$ et $h$ jouent au ping-pong sur l'espace $X$. On appelle $V_g$ l'espace attractif de $g$ et $H_g$ l'espace r\'{e}pulsif. Idem pour $g^{-1}$, $h$ et $h^{-1}$.
\label{ping} \end{lemme}
Gr\^{a}ce \`{a} ce lemme, le Th\'{e}or\`{e}me \ref{titsproba} se d\'{e}duit
imm\'{e}diatement de l'\'{e}nonc\'{e} suivant:
\begin{theorem}
Soient $\Gamma_1$ et $\Gamma_2$ deux sous-groupes Zariski denses de $G$ munis chacun d'une mesure de probabilit\'{e} adapt\'{e}e ayant un moment exponentiel qu'on notera respectivement $\mu_1$ et $\mu_2$.  Soient $(g_i)_{i=1\in \NN^*}$ (resp.$(g'_i)_{i=1\in \NN^*}$) une suite de variables ind\'{e}pendantes de loi $\mu_1$ (resp. $\mu_2$). On suppose que les deux suites sont ind\'{e}pendantes entre elles. On note $x_n=g_1\cdots g_n$ et $x'_n=g'_1\cdots g'_n$ les marches al\'{e}atoires \`{a} droite respectivement associ\'{e}es. \\
Alors, avec probabilit\'{e} tendant vers $1$ de fa\c{c}on exponentielle, $x_n$ et $x'_n$ jouent au ping-pong sur la fronti\`{e}re de Furstenberg.  En particulier, presque s\^{u}rement, pour $n$ assez grand, le groupe $\langle x_n, x'_n\rangle $ engendr\'{e} par $x_n$ et $x'_n$ est libre.\label{titsproba1}\end{theorem}

\begin{remarque} Dans \cite{Guivarch3}, Guivarc'h d\'{e}montre (pour $\Gamma_1=\Gamma_2$ et $\mu_1=\mu_2$) qu'il existe presque s\^{u}rement des sous-suites $n_{k},n'_{k}$ et des entiers $p_k,p'_k$ tels que $\langle {x_{n_k}}^{p_k},{x'_{n'_k}}^{p'_k}\rangle$ est libre pour $k$ assez grand . Ceci est d\'{e}j\`{a} une preuve probabiliste de l'alternative de Tits.  Un point crucial pour  arriver \`{a} se d\'{e}barrasser de ces sous-suites et de ces entiers est d'utiliser la d\'{e}composition $KA^+K$ pour avoir la contraction, et non les valeurs propres. Un autre point important est la proximalit\'{e}, i.e. l'estimation de la distance entre les espaces attractifs et r\'{e}pulsifs d'une m\^{e}me marche. Ce point d\'{e}licat  est une cons\'{e}quence de l'ind\'{e}pendance asymptotique. \end{remarque}
\paragraph{Notations}
Soit $B=G/P$ la fronti\`{e}re de Furstenberg de $G$ (voir la D\'{e}finition \ref{deffront}) et $B \hookrightarrow \prod_{i=1}^r P(V_i)$, $B^*\hookrightarrow \prod_{i=1}^r P(V_i^*)$ les plongements de $B$ et $B^*$ dans des produits d'espaces projectifs de repr\'{e}sentations irr\'{e}ductibles et proximales (voir le Lemme \ref{furstprojective}).\\
Pour tout $y\in B^*$, on note $$\ker(y)=\{x=([x_1],\cdots,[x_r])\in \prod_{i=1}^r P(V_i); \exists i\in \{1,\cdots, r\}; x_i\in \ker(y_i)\}$$
Pour tout $g\in G$, on note $v_{g}=k(g)M\in B$ et
$\widetilde{H_g} = \ker \left( M u(g) \right)\in\prod_{i=1}^r P(V_i)$.\\

\vspace{1cm}

L'espace attractif de la marche al\'{e}atoire $x_n$ sera une petite  boule  autour de $v_{x_n}$ et l'espace attractif une petite boule de $B$ autour de $\widetilde{H_{x_n}}$. Idem pour les autres joueurs de ping-pong.

\noindent Le Th\'{e}or\`{e}me \ref{titsproba1} d\'{e}coule imm\'{e}diatement du Lemme \ref{ping} et  des Propositions \ref{proposition1}, \ref{proposition2} et \ref{proposition3} suivantes.
\begin{proposition}(Contraction)\\
Pour tout entier $n$ et tout $\epsilon>0$, on note $A_n{(\epsilon)}$ l'\'{e}v\'{e}nement $[x\in B, \;  \delta(x,\widetilde{H_{x_n}})\geq \epsilon \Longrightarrow \delta(x_n\cdot x, v_{x_n})\leq \epsilon]$. Alors il existe $\rho\in ]0,1[$ tel que pour tout entier $n$ assez grand:

$$\lim \frac{1}{n} \log \PP \left(\Omega \setminus A_n(\rho^n) \right)  < 0$$
En d'autres termes, avec probabilit\'{e} tendant vers $1$ exponentiellement vite, le point $1$ du lemme de ping-pong est v\'{e}rifi\'{e} avec $V_{x_n}$ la $\rho^n$-boule de centre $v_{x_n}$ et $H_{x_n}$ l'intersection entre $B$ et la $\rho^n$-boule autour de l'espace $\widetilde{H_{x_n}}$.\\
Une assertion similaire est valable pour les marches ${x_n}^{-1}$, ${x'}_n$ et ${x'_n}^{-1}$.
\label{proposition1}\end{proposition}

\begin{proposition}(Proximalit\'{e})\\
Pour tout $t\in ]0,1[$,
$$\underset{n\rightarrow \infty}{\lim}\frac{1}{n} \log \PP \left(\delta(v_{x_n}, \widetilde{H_{x_n}}) \leq t^n \right)< 0$$
Une assertion similaire est valable en rempla\c{c}ant $x_n$ respectivement par ${x_n}^{-1}$, ${x'}_n$ et ${x'_n}^{-1}$.\\
En particulier, le point $2$ du lemme de ping-pong est v\'{e}rifi\'{e}.   \label{proposition2}\end{proposition}

\begin{proposition}(Espaces en position g\'{e}n\'{e}rale)\\
Pour tout $t\in ]0,1[$,
$$\underset{n\rightarrow \infty}{\lim} \frac{1}{n} \log \PP \left(\delta(v_{x_n}, \widetilde{H_{x'_n}}) \leq t^n \right)< 0$$
Une assertion similaire est valable en gardant $v_{x_n}$ et en rempla\c{c}ant $\widetilde{H_{x'_n}}$ par l'espace r\'{e}pulsif de  ${x'_n}^{-1}$. Les r\^{o}le de $x_n$ et $x'_n$ sont interchangeables.\\
En particulier, le point $3$ du lemme de ping-pong est v\'{e}rifi\'{e}.
\label{proposition3}\end{proposition}

Pour prouver la Proposition \ref{proposition1}, nous commen\c{c}ons par un lemme simple mais crucial. Gr\^{a}ce \`{a} ce dernier, nous pouvons avoir la contraction et faire du ping-pong sans avoir recours aux valeurs propres, mais plut\^{o}t aux valeurs singuli\`{e}res i.e. \`{a} la d\'{e}composition $G=KA^+K$.
\begin{lemme}Soit $(\rho,V)$ une repr\'{e}sentation rationnelle irr\'{e}ductible de $G$. On consid\`{e}re la base $F=(e_1,\cdots, e_d)$  et le produit scalaire obtenus par le Corollaire \ref{mostowcons}. On note $\delta$ la distance de Fubini-Study  sur l'espace projectif $P(V)$ (D\'{e}finition \ref{distfront}). Pour tout $a\in A$,  la matrice diagonale $\rho(a)$  est not\'{e}e dans cette base par $diag(a_1,\cdots, a_d)$. On pose $v_{g,\rho}=k(g)\cdot[e_1]$ et $\widetilde{H_{g,\rho}}$ l'hyperplan noyau de la forme lin\'{e}aire $u(g)^{-1}\cdot e_1^*$.\\
Si $\underset{j\neq 1}{\sup} \frac{a_j(g)}{a_1(g)} < \epsilon^2$ alors $g$ est ${\epsilon}$-contractant dans le sens suivant: $$\delta([x],\widetilde{H_{g,\rho}})\geq  \epsilon \Longrightarrow \delta([x],v_{g,\rho}) \leq \epsilon$$\end{lemme}
\noindent La preuve est un exercice d'alg\`{e}bre lin\'{e}aire. Nous r\'{e}f\'{e}rons \`{a} \cite[Proposition 3.1]{breuillard} pour une preuve o\`{u} une sorte de r\'{e}ciproque est d\'{e}montr\'{e}e.
\begin{proof}[Preuve   de la Proposition \ref{proposition1}:]
On rappelle que $B \hookrightarrow \prod_{i=1}^r P(V_i)$. La Proposition \ref{propratio} montre que pour tout $i$, il existe $\eta_i\in ]0,1[$ tel que $\rho_i(x_n)$ est $\eta_i^n$-contractant sur un ensemble de probabilit\'{e} exponentiellement proche de $1$. On peut alors appliquer  le lemme précédent pour la représentation $(\rho_i,V_i)$. \end{proof}

\begin{proof}[Preuve de la Proposition \ref{proposition2}:]
Soit $\phi_1$ la fonction r\'{e}elle  d\'{e}finie sur $B \times B^*$ par: $\phi_1(x,y)=\delta(x,\ker(y))$.
%\underset{\alpha \in \Pi}{\inf} \delta \left((x, \ker (\pi_\alpha(y)) \right)$.
On peut v\'{e}rifier que $\phi_1$ est lipschitzienne.  De plus, pour tout $\epsilon>0$,  on peut construire une fonction lipschitzienne  $\phi_{2,\epsilon}: \RR \longrightarrow \RR$ telle que
$\mathds{1}_{[-\epsilon; \epsilon]}\leq \phi_{2,\epsilon} \leq \mathds{1}_{[-2\epsilon; 2\epsilon]}$ et ayant une constante de Lipschitz de l'ordre de $\frac{1}{\epsilon}$. On pose $\phi_{\epsilon}=\phi_{2,\epsilon}\circ \phi_1$, qui est lipschitzienne  de constante de Lipschitz de l'ordre de $\frac{1}{\epsilon}$. Avec ces notations, nous avons pour tout $t\in ]0,1[$:
$$\PP \left( \delta(v_{x_n}, \widetilde{H_{x_n}}) \leq t^n \right) = \EE \big[ \phi_{t^n}\left(k(x_n)M, Mu(x_n)\right)\big]$$
 Le Th\'{e}or\`{e}me \ref{indasymp} d'ind\'{e}pendance asymptotique montre  qu'il existe $\rho\in ]0,1[$ et  deux variables al\'{e}atoires ind\'{e}pendantes $Z\in B$ et $T\in B^*$  telles que:
 $$\PP \left( \delta(v_{x_n}, \widetilde{H_{x_n}}) \leq t^n \right) \preceq \EE \left( \phi_{t^n}(Z, T)\right) + \frac{\rho^n}{t^n}$$
Sans perte de g\'{e}n\'{e}ralit\'{e}, on peut supposer que $t\in ]\rho, 1[$.  Or,
$$\EE \left( \phi_{t^n}(Z, T)\right) \leq \PP \left(  \delta(Z, \ker (T)) \leq 2 t^n \right)\leq  \sum_{i=1}^r \PP\left(\delta(Z_i, \ker(T_i))\leq 2 t^n \right) $$
L'hyperplan $\ker (T_i) $ de $V_i$ est  al\'{e}atoire  mais ind\'{e}pendant de $Z_i$. Comme  $Z_i$ a pour loi l'unique mesure de probabilit\'{e} $\mu$-invariante sur l'espace projectif $P(V_i)$, alors l'uniformit\'{e} dans le Th\'{e}or\`{e}me \ref{hausdorff1} permet de montrer que la quantit\'{e} pr\'{e}c\'{e}dente d\'{e}cro\^{\i}t exponentiellement vite vers z\'{e}ro.
  \end{proof}

\begin{proof}[Preuve de la Proposition \ref{proposition3}:]

D'apr\`{e}s le Th\'{e}or\`{e}me \ref{convkak}, le point attractif $v_{x_n}=k(x_n)M$ converge presque s\^{u}rement vers une variable al\'{e}atoire $Z$ et de fa\c{c}on exponentielle. Par l'in\'{e}galit\'{e} triangulaire, il suffit donc de d\'{e}montrer que pour tout $t\in ]0,1[$, la probabilit\'{e}  $\PP \left(\delta(Z,\widetilde{H_{x'_n}})\leq t^n \right)$ d\'{e}cro\^{\i}t exponentiellement vite vers $0$. Ceci est garanti gr\^{a}ce \`{a} l'ind\'{e}pendance entre $Z$ et $\widetilde{H_{x'_n}}$ et le Th\'{e}or\`{e}me \ref{hausdorff1}.   \end{proof}
\paragraph{Probl\`{e}me ouvert:} Il est int\'{e}ressant de trouver des sous-groupes libres v\'{e}rifiant certaines contraintes (voir par exemple \cite{fortementdense}). Un probl\`{e}me dans cette lign\'{e}e serait de savoir si le sous-groupe g\'{e}n\'{e}rique $\langle x_n,x'_n\rangle$ trouv\'{e} ci-dessus intersecte trivialement toute vari\'{e}t\'{e} alg\'{e}brique propre de $G$? Cela donnerait un sous-groupe libre de $\Gamma$ dont l'intersection avec toute vari\'{e}t\'{e} alg\'{e}brique propre est triviale. Des r\'{e}sultats proches mais moins forts sont contenus dans \cite{aoun1}.

%  \newcommand{\Addresses}{{% additional braces for segregating \footnotesize
 % \bigskip
%  \footnotesize

 % Richard~Aoun, \textsc{American University of Beirut, Department of Mathematics,  
%Faculty of Arts and Sciences, 
 %P.O. Box 11-0236 
%Riad El Solh,
%Beirut 1107 2020, 
%LEBANON}\par\nopagebreak
 
 % }}

%\Addresses

\end{document}